\documentclass[aos,MSNbibl,nameyear,seceqn,dvips]{arximspdf}
\usepackage{graphicx}
\doi{10.1214/12-AOS1066} %kopijuoti is PTS
\volume{40}
\issue{6}
\pubyear{2012}
\firstpage{3108}
\lastpage{3136}

\makeatletter
\def\mid{|}
\newtheorem{theorem}{Theorem}[section]
\newtheorem{lemma}{Lemma}[section]
\newtheorem{corollary}{Corollary}[section]
\newtheorem{proposition}{Proposition}[section]
\newproclaim{remark}{Remark}[section]
\newproclaim{example}{Example}[section]
\newproclaim{notation}{Notation}

\newcommand{\Cov}{\operatorname{Cov}}
\newcommand{\Card}{\operatorname{Card}}
\newcommand{\AIC}{\mathtt{AIC}}
\newcommand{\BIC}{\mathtt{BIC}}
\newcommand{\GACV}{\mathtt{GACV}}
\newcommand{\IAIC}{\mathtt{IAIC}}
\newcommand{\IBIC}{\mathtt{IBIC}}
\newcommand{\IGACV}{\mathtt{IGACV}}
\newcommand{\QAMISE}{\mathtt{QA\mbox{-}MISE}}

\newcommand{\E}{\mathbb{E}}
\newcommand{\Prob}{\mathbb{P}}
\newcommand{\En}{\mathbb{E}_{n}}
\newcommand{\R}{\mathbb{R}}
\newcommand{\T}{\mathcal{T}}
\newcommand{\U}{\mathcal{U}}
\newcommand{\indep}{\mathop{\perp\!\!\!\!\perp}}
\makeatother

\begin{document}
\begin{frontmatter}

\title{Estimation in functional linear quantile~regression\thanksref{T1}}
\runtitle{Functional quantile regression}

\thankstext{T1}{Supported by the Grant-in-Aid for Young Scientists (B) (22730179) from the JSPS.}

\begin{aug}
\author[A]{\fnms{Kengo}~\snm{Kato}\corref{}\ead[label=e1]{kkato@hiroshima-u.ac.jp}}
\runauthor{K. Kato}
\affiliation{Hiroshima University}
\address[A]{Department of Mathematics \\
Graduate School of Science\\
Hiroshima University\\
1-3-1 Kagamiyama, Higashi-Hiroshima \\
Hiroshima 739-8526\\
Japan\\
\printead{e1}} %adresu isvedimo komanda gale!
\end{aug}

% HISTORY:
\received{\smonth{3} \syear{2012}}
\revised{\smonth{8} \syear{2012}}

% ABSTRACT
%
\begin{abstract}
This paper studies estimation in functional linear quantile regression
in which the dependent variable is scalar while the covariate is a
function, and the conditional quantile for each fixed quantile index is
modeled as a linear functional of the covariate. Here we suppose that
covariates are discretely observed and sampling points may differ
across subjects, where the number of measurements per subject increases
as the sample size.
Also, we allow the quantile index to vary over a given subset of the
open unit interval, so the slope function is a function of two
variables: (typically) time and quantile index. Likewise, the
conditional quantile function is a function of the quantile index and
the covariate.
We consider an estimator for the slope function based on the principal
component basis.
An estimator for the conditional quantile function is obtained by a
plug-in method.
Since the so-constructed plug-in estimator not necessarily satisfies
the monotonicity constraint with respect to the quantile index,
we also consider a class of monotonized estimators for the conditional
quantile function.
We establish rates of convergence for these estimators under suitable
norms, showing that these rates are optimal in a minimax sense under
some smoothness assumptions on the covariance kernel of the covariate
and the slope function.
Empirical choice of the cutoff level is studied by using simulations.
\end{abstract}

% KEYWORDS
% Pirmas kwd is didziosios raides
%

\end{frontmatter}

%s1 #&#
\section{Introduction}\label{sec1}

Quantile regression, initially developed by the seminal work of \citet
{KB78}, is one of the most important statistical methods in measuring
the impact of covariates on dependent variables. An attractive feature
of quantile regression is that it allows us to make inference on the
entire conditional distribution by estimating
several different conditional quantiles. Some basic materials on
quantile regression and its applications are summarized in \citet{K05}.

This paper studies estimation in functional linear quantile regression
in which the dependent variable is scalar while the covariate is a
function, and the conditional quantile for each fixed quantile index is
modeled as a linear functional of the covariate. The model that we
consider is an extension of functional linear regression to the
quantile regression case. Here we suppose that covariates are
discretely observed and sampling points may differ across subjects,
where the number of measurements per subject increases as the sample size.
Also, we allow the quantile index to vary over a given subset of the
open unit interval, so the slope function is a function of two
variables: (typically) time and quantile index. Likewise, the
conditional quantile function is a function of the quantile index and
the covariate.
We consider the problem of estimating the slope function as well as the
conditional quantile function itself.
The estimator we consider for the slope function is based on the
principal component analysis (PCA).
Expanding the covariate and the slope function in terms of the PCA
basis, the model is transformed into a quantile regression model with
an infinite number of regressors.
Truncating the infinite sum by the first $m$ (say) terms, we may apply
a standard quantile regression technique to estimating the first $m$
coefficients in the basis expansion of the slope function at each
quantile index, where $m$ diverges as the sample size. In practice, the
population PCA basis is unknown, so it has to be replaced by a suitable
estimate for it. Once the estimator for the slope function is
available, an estimator for the conditional quantile function is
obtained by a plug-in method.
Since the so-constructed plug-in estimator not necessarily satisfies
the monotonicity constraint with respect to the quantile index,
we also consider a class of monotonized estimators for the conditional
quantile function.

In summary, we have the following three types of estimators in mind:
\begin{longlist}[(iii)]
\item[(i)] a PCA-based estimator for the slope function;
\item[(ii)] a plug-in estimator for the conditional quantile function;
\item[(iii)] monotonized estimators for the conditional quantile function.
\end{longlist}
We establish rates of convergence for these estimators under suitable
norms, showing that these rates are optimal in a minimax sense under
some smoothness assumptions on the covariance kernel of the covariate
and the slope function.

In practice, we have to choose the cutoff level empirically. We suggest
some criteria, namely, (integrated-)AIC, BIC and GACV, to choose the
cutoff level.
We study the performance of these criteria by using simulations. In our
limited simulation experiments, although none of these criteria clearly
dominated the others, (integrated-)BIC worked relatively stably.

Functional data have become increasingly important.
For data collected on dense grids, a traditional multivariate analysis
is not directly applicable since the number of grid points may be
larger than the sample size and the correlation\vadjust{\goodbreak} between distinct grid
points is potentially high.
\textit{Functional data analysis} views such data as realizations of
random functions and takes into account the functional nature of the data.
We refer the reader to \citet{RS05} for a comprehensive treatment on
functional data analysis.
Earlier theoretical studies in functional data analysis have focused on
functional linear mean regression models [see \citet{CH06},
\citeauthor{CFS99} (\citeyear{CFS99,CFS03}), \citet{YMW05,HH07,CKS09,JWZ09,YC10,DH11} and references
cited in these papers].\vadjust{\goodbreak} Among them, \citet{HH07} established fundamental
results in functional linear mean regression, deriving sharp rates of
convergence for a PCA-based estimator for the slope function under some
smoothness assumptions. Note that they assumed that covariates are
continuously observed.
Other than functional linear mean regression, \citet{MS05} developed
estimation methods for generalized functional linear models using
series expansions of covariates and slope functions.
% See also \citet{DPZ10}.\footnote{The techniques of their paper are not
%applicable here since among others the criterion function is not
%smooth in the quantile regression case and hence the overall proof
%strategies are significantly different. }

While not many, there are some earlier papers on estimating conditional
quantiles with function-valued covariates.
\citet{CCS05} studied smoothing splines estimators for functional linear
quantile regression models, while their established rates are not sharp.
\citet{FRV05} considered nonparametric estimation of conditional
quantiles when covariates are functions, which is a different topic
than ours.
\citet{CM11} considered an ``indirect'' estimation of conditional
quantiles. They modeled the conditional distribution as the composition
of some (possibly unknown) link function and a linear functional of the
covariate. They first estimated the conditional distribution function
by adapting the method developed in \citet{MS05} and then estimated the
conditional quantile function by inverting the estimated conditional
distribution function. In the quantile regression literature, there are
two ways to estimate conditional quantiles. One is to directly model
conditional quantiles and estimate unknown parameters by minimizing
check functions. The other is to estimate conditional distribution
functions and invert them to estimate conditional quantile functions.
We refer to the former as a ``direct'' method while to the latter as an
``indirect'' method. The approach taken in this paper is classified
into a direct method, while that of \citet{CM11} is classified into an
indirect method. Note that although their method is flexible, they only
established consistency of the estimator.
% Aside from theory, functional quantile regression has been applied in
%analysis of ozone pollution data \citep{CCS07}, EL Ni\~{n}o data

Conditional quantile estimation offers a variety of fruitful
applications for data containing function-valued covariates.
A leading example in which conditional quantile estimation with
function-valued covariates is useful appears in analysis of growth data
[\citet{CM11}]. Suppose that we have a growth data set of girls' heights
between ages $1$ and $18$, say, where multiple measurements may occur
at some ages.
Use girl's growth history between ages $1$ and $12$ as a covariate, and\vadjust{\goodbreak}
her height at age $18$ as a response.
Then, conditional quantile estimation gives us an overall picture of
the predictive distribution of girl's height at age $18$
given her growth history between ages $1$ and $12$, which is more
informative than just knowing the mean response. In addition to growth
data, functional quantile regression has been applied in analysis of
ozone pollution data [\citet{CCS07}] and El Ni\~{n}o data [\citet
{FRV05}]. We believe that functional linear quantile regression
modeling is a benchmark modeling in conditional quantile estimation
when covariates are functions, just as linear quantile regression
modeling is so when covariates are vectors.

Our estimator for the slope function (at a fixed quantile index) can be
understood as a regularized solution to an empirical version of a {\em
nonlinear} ill-posed inverse problem that corresponds to the ``normal
equation'' in the quantile regression case, where the regularization is
controlled by the cutoff level. The paper is thus in part related to
the literature on statistical nonlinear inverse problems, which is
still an ongoing research area [see \citet{BHM04,HL07,LC10,CP11,GS11}].
On the other hand, in the mean regression case, the normal equation
becomes a \textit{linear} ill-posed inverse problem. \citet{HH07}
considered two regularized estimators for the slope function based on
the normal equation in the mean regression case.
Conceptually, the problems handled in our and their papers are
different in their nature: linearity and nonlinearity.

From a technical point of view, establishing sharp rates of convergence
for our estimators is challenging.
Our proof strategy builds upon the techniques developed in the
asymptotic analysis for $M$-estimators with diverging numbers of
parameters [see, e.g., \citet{HS00,BCF11}]. However, the additional
complication arises essentially because ``regressors'' here are
estimated ones and the estimation error has to be taken into account,
which requires some new techniques such as ``conditional'' maximal
inequalities and some careful moment calculations. Additionally, the
discretization error brings a further technical complication.

Finally, the setting here is similar to Section 3 of \citet{CKS09}:
covariates are densely but discretely observed, and the discretization
error is taken into account in the analysis. However, the paper does
not cover the case in which covariates are discretely observed with
measurement errors because of the technical complication. A formal
theoretical analysis in such a case is left to a future work.

The remainder of the paper is organized as follows. Section \ref
{method} presents the model and estimators. Section~\ref{main} gives
the main results in which we derive rates of convergence for the
estimators. Section~\ref{MC} discusses empirical choice of the cutoff level.
A~proof of Theorem~\ref{thm1} is given in Section~\ref{proof1}.
Some additional discussions, technical proofs, useful technical tools
and simulation results are provided in the Appendix. Due to the space
limitation, the Appendix is contained in the supplementary file [\citet{supp}].

\begin{notation*} Throughout the paper, we shall obey the following notation.
For $z \in\mathbb{R}^{k}$, let $\| z \|_{\ell^{2}}$ denote the
Euclidean norm of $z$. For any integer $k \geq2$, let $\mathbb
{S}^{k-1}$ denote the set of all unit vectors in $\mathbb{R}^{k}$:
$\mathbb{S}^{k-1} = \{ z \in\mathbb{R}^{p} \dvtx\| z \|_{\ell^{2}}
= 1 \}
$. For any $y,z \in\mathbb{R}$, let $y \vee z = \max\{ y,z \}$ and $y
\wedge z = \min\{ y,z \}$.
Let $1(\cdot)$ denote the indicator function. For any given (random or
nonrandom, scalar or vector) sequence $\{ z_{i} \}_{i=1}^{n}$, we use
the notation
\[
\En[ z_{i} ] = \frac{1}{n} \sum_{i=1}^{n}
z_{i},
\]
which should be distinguished from the population expectation $\E
[\cdot
]$. For any two sequences of positive constants $r_{n}$ and $s_{n}$, we write
$r_{n} \asymp s_{n}$ if the ratio $r_{n}/s_{n}$ is bounded and bounded
away from zero. Let $L_{2}[0,1]$ denote the usual $L_{2}$ space with
respect to the Lebesgue measure for functions defined on $[0,1]$. Let
$\| \cdot\|$ denote the $L_{2}$-norm: $\| f \|^{2} = \int_{0}^{1}
f^{2}(t) \,dt$. For any finite set $I$, $\Card(I)$ denotes the
cardinality of $I$.
\end{notation*}

%s2 #&#
\section{Methodology}
\label{method}

%s2.1 #&#
\subsection{Functional linear quantile regression modeling}\label{sec2.1}

Let $(Y,X)$ be a pair of a scalar random variable $Y$ and a random
function $X=(X(t))_{t \in\T}$ on a bounded closed interval $\T$ in
$\mathbb{R}$.
Without loss of generality, we assume $\T= [0,1]$.
By ``random function,'' we mean that $X(t)$ is a random variable for
each $t \in[0,1]$.
For our purpose of formulating a functional linear quantile regression
model, we need the existence of the regular conditional distribution of
$Y$ given $X$, to which end we assume a mild regularity condition on
the path property of $X$.
Let $D[0,1]$ denote the space of all c\`{a}dl\`{a}g functions on
$[0,1]$, equipped with the Skorohod metric [see \citet{B68}].
We assume that the map $t \mapsto X(t)$ is c\`{a}dl\`{a}g almost
surely. Equip $D[0,1]$ with the Borel $\sigma$-field. Then, $X$ can be
taken as a $D[0,1]$-valued random variable.
Since $D[0,1]$ is a Polish space, and the product space $\mathbb{R}
\times D[0,1]$ with the product metric is also Polish, the regular
conditional distribution of $Y$ given $X$ exists [see, e.g., \citet{Du02}, Theorem 10.2.2].

Let $Q_{Y|X}(\cdot\mid X)$ denote the conditional quantile function of
$Y$ given $X$. Let $\U$ be a given subset of $(0,1)$ that is away from
$0$ and $1$, that is, for some small $\epsilon\in(0,1/2)$, $\U
\subset[\epsilon, 1-\epsilon]$.
For each $u \in\U$, we assume that $Q_{Y | X}(u \mid X)$ can be
written as a linear functional of $X$, that is, for each $u \in\U$,
there exist a scalar constant $a(u) \in\R$ and a scalar function $
b(\cdot,u) \in L_{2}[0,1]$ such that
%
%e2.1 #&#
%
\begin{equation}
Q_{Y | X} (u \mid X) = a(u) + \int_{0}^{1}
b(t,u) X^{c}(t) \,dt,\qquad u \in\U, \label{model}
\end{equation}
where $X^{c} (t) = X(t) - \E[ X(t) ]$.
Typical examples of $\U\subset(0,1)$ are as follows:
(i)~$\U= \{ u \}$ (singleton); (ii) $\U= \{ u_{1},\ldots, u_{K} \}$
with \mbox{$0< u_{1} < \cdots< u_{K} < 1$} (finite set); (iii) $\U= [
u_{L},u_{U} ]$ with $0 < u_{L} < u_{U} < 1$ (bounded closed interval).
Formally, we allow for all these possibilities.

The model (\ref{model}) is a natural extension of standard linear
quantile regression models to function-valued covariates and was first
formulated in \citet{CCS05}.
In what follows, we consider estimating the slope function $(t,u)
\mapsto b(t,u)$ and the conditional quantile function $(u,x) \mapsto
Q_{Y|X}(u \mid x)$.

Before going to the next step, we briefly give two simple examples of
data generating processes that admit the conditional quantile
restriction (\ref{model}).

%ex2.1 #&#
%
\begin{example}[(Linear location design)] Suppose that $(Y,X)$ obeys the
linear location design
\[
Y = c + \int_{0}^{1} \varrho(t) X(t) \,dt +
\varepsilon, \qquad \varepsilon\indep X,
\]
where $c$ is a constant and $\varrho(t)$ is a function in $L_{2}[0,1]$.
Let $F_{\varepsilon}$ denote the distribution function of $\varepsilon$
and denote by $F_{\varepsilon}^{-1}$ the quantile function of
$F_{\varepsilon}$. Then $(Y,X)$ obeys the conditional quantile restriction
\[
Q_{Y | X} (u \mid X ) = c + F^{-1}_{\varepsilon}(u) + \int
_{0}^{1} \varrho(t) X(t) \,dt, \qquad u \in(0,1).
\]
\end{example}

%ex2.2 #&#
%
\begin{example}[(Linear location-scale design)] Suppose that $(Y,X)$
obeys the linear location-scale design
\[
Y = c_{1} + \int_{0}^{1}
\varrho_{1} (t) X(t) \,dt + \sigma(X) \varepsilon,\qquad \sigma(X) =
c_{2} + \int_{0}^{1}
\varrho_{2}(t) X(t) \,dt, \varepsilon\indep X,
\]
where $c_{1},c_{2}$ are constants and $\varrho_{1}(t),\varrho_{2}(t)$
are functions in $L_{2}[0,1]$. Suppose that $\sigma(X) > 0$ almost
surely. Denote by $F_{\varepsilon}^{-1}$ the quantile function of the
distribution of $\varepsilon$. Then $(Y,X)$ obeys the conditional
quantile restriction
\[
Q_{Y | X} (u \mid X ) = c_{1} + c_{2}
F^{-1}_{\varepsilon}(u) + \int_{0}^{1}
\bigl\{ \varrho_{1} (t) + \varrho_{2}(t) F_{\varepsilon}^{-1}(u)
\bigr\} X(t) \,dt, \qquad u \in(0,1).
\]
In this design, the slope function depends on the quantile index.
\end{example}

%s2.2 #&#
\subsection{Estimation strategy}\label{sec2.2}

We base our estimation strategy on the principal component analysis
(PCA). To this end, we additionally assume here that $\int_{0}^{1} \E[
X^{2}(t) ] \,dt < \infty$.
Define the covariance kernel $K(s,t) = \Cov(X(s),\break X(t))$. Then, by the
Hilbert--Schmidt theorem,
$K(s,t)$ admits the spectral expansion
\[
K(s,t) = \sum_{j=1}^{\infty}
\kappa_{j} \phi_{j}(s) \phi_{j}(t),
\]
where $\kappa_{1} \geq\kappa_{2} \geq\cdots\geq0$ are ordered
eigenvalues, and $\{ \phi_{j} \}_{j=1}^{\infty}$ is an orthonormal
basis of $L_{2}[0,1]$ consisting of eigenfunctions of the integral
operator from $L_{2}[0,1]$ to itself with kernel $K(s,t)$. We will
later assume that $\kappa_{j}$ are all positive and there are no ties
in $\kappa_{j}$, that is, $\kappa_{1} > \kappa_{2} > \cdots> 0$.
Since $\{ \phi_{j} \}_{j=1}^{\infty}$ is an orthonormal basis of
$L_{2}[0,1]$, we have the following expansions in $L_{2}[0,1]$:
\[
X^{c}(t) = \sum_{j=1}^{\infty}
\xi_{j} \phi_{j}(t),\qquad b(t,u) = \sum
_{j=1}^{\infty} b_{j}(u) \phi_{j}(t),
\]
where $\xi_{j}$ and $b_{j}(u)$ are defined by
\[
\xi_{j} = \int_{0}^{1}
X^{c}(t) \phi_{j}(t) \,dt,\qquad b_{j}(u) = \int
_{0}^{1} b(t,u) \phi_{j}(t) \,dt.
\]
The $\xi_{j}$ are called ``principal scores'' for $X$. The expansion
for $X^{c}$ is called the ``Karhunen--Lo\`{e}ve expansion.''
This leads to the expression $\int_{0}^{1} b(t,u) X^{c}(t)\,dt = \sum_{j=1}^{\infty} b_{j}(u)\xi_{j}$.
Then, the model (\ref{model}) is transformed into a quantile regression
model with an infinite number of ``regressors'':
%
%e2.2 #&#
%
\begin{equation}
Q_{Y|X} (u \mid X) = a(u) + \sum_{j=1}^{\infty}
b_{j} (u) \xi_{j},\qquad u \in\U. \label{qmodel}
\end{equation}
Note that $\E[ \xi_{j} ] = 0, \E[ \xi_{j}^{2} ] = \kappa_{j}$ and
$\E[
\xi_{j} \xi_{k} ] = 0$ for all $j \neq k$.

We first consider estimating the slope function $(t,u) \mapsto b(t,u)$.
To this end, we estimate the function $b(\cdot,u)$ for each $u \in\U$
and collect them to construct a final estimator for $(t,u) \mapsto b(t,u)$.
To explain the basic idea, suppose for a while that (i) $X$ were
continuously observable; and (ii) the covariance kernel $K(s,t)$ were known.
The problem then reduces to finding suitable estimates of the
coefficients $b_{j}(u)$.
Let $(Y_{1},X_{1}),\ldots,(Y_{n},X_{n})$ be independent copies of
$(Y,X)$. For each $i=1,\ldots,n$, let $\xi_{ij}$ be the principal scores
for $X_{i}$.
Pick any $u \in\U$. Then, a plausible approach to estimating $b(\cdot
,u)$ is to truncate $\sum_{j=1}^{\infty} b_{j}(u) \xi_{j}$ by $\sum_{j=1}^{m} b_{j}(u) \xi_{j}$ for some large $m$, and estimate only the
first $m$ coefficients $b_{1}(u),\ldots,b_{m}(u)$ using a standard
quantile regression technique.
Let $m = m_{n}$ be the ``cutoff'' level such that $1 \leq m \leq n-1$
and $m \to\infty$ as $n \to\infty$. Estimate $a(u)$ and the first $m$
coefficients $b_{1}(u),\ldots,b_{m}(u)$ of $b(\cdot,u)$ by
%
%e2.3 #&#
%
\begin{equation}
\qquad \bigl(\tilde{a}(u),\tilde{b}_{1}(u),\ldots,
\tilde{b}_{m}(u) \bigr) = \mathop{\arg\min}_{a,b_{1},\ldots,b_{m}} \En \Biggl[
\rho_{u} \Biggl(Y_{i} - a - \sum
_{j=1}^{m} b_{j} \xi_{ij}
\Biggr) \Biggr], \label{inf}
\end{equation}
where $\rho_{u} (y) = \{ u - 1(y \leq0) \} y$ is the check function
[\citet{KB78}]. Note that for $u=0.5$, $\rho_{0.5}(\cdot)$ is equivalent
to the absolute value function. Here recall that $\En[ z_{i} ] =
n^{-1} \sum_{i=1}^{n} z_{i}$ for any sequence $\{ z_{i} \}_{i=1}^{n}$.
The resulting estimator for the slope function $(t,u) \mapsto b(t,u)$
is given by
\[
\tilde{b}\dvtx(t,u) \mapsto\tilde{b}(t,u),\qquad \tilde{b}(t,u) = \sum
_{j=1}^{m} \tilde{b}_{j}(u)
\phi_{j}(t), \qquad t \in[0,1], u \in\U.
\]

However, this ``estimator'' is infeasible since (i) $X$ is usually
discretely observed; and (ii) $K(s,t)$ is unknown.
Because of (i), it is usually not possible to directly estimate the
covariance kernel $K(s,t)$ by the empirical one (since $\En[
(X_{i}(s)-\bar{X}(s))(X_{i}(t) - \bar{X}(t))]$ with $\bar{X}(t) =
n^{-1} \sum_{i=1}^{n} X_{i}(t) $ is unavailable for some $(s,t)$).
Similarly to \citet{CKS09}, we consider the following setting:
\begin{longlist}[(1)]
\item[(1)] For each $i=1,\ldots,n$, $X_{i}$ is observed only at
$L_{i}+1$ discrete points $0=t_{i1} < t_{i2} < \cdots< t_{i,L_{i}+1} =
1$, that is, we only observe $X_{i}(t_{ij}), j=1,\ldots,L_{i}+1$.
Typically, $\max_{1 \leq i \leq n} \max_{1 \leq l \leq L_{i}} (
t_{i,l+1} - t_{il} ) \to0$ as $n \to\infty$ is assumed.
\item[(2)] Based on the discrete observations, for each $i = 1,\ldots,
n$, construct an interpolated function $\hat{X}_{i} = (\hat
{X}_{i}(t))_{t \in[0,1]}$ for $X_{i}=(X_{i}(t))_{t \in[0,1]}$.
\end{longlist}
Here we shall use a simple interpolation rule (see also the later remark):
\[
\hat{X}_{i}(t) = \sum_{l=1}^{L_{i}}
X(t_{il}) 1 \bigl(t \in[t_{il},t_{i,l+1} )\bigr),
\qquad i=1,\ldots,n.
\]
The observation time points $t_{i1},\ldots,t_{i,L_{i}+1}$ (and $L_{i}$)
should be indexed by the sample size $n$, however, this is suppressed
for the notational convenience.
Suppose now that the interpolated functions $\hat{X}_{1},\ldots,\hat
{X}_{n}$ are obtained. Then, we may estimate the covariance kernel
$K(s,t)$ by
\[
\hat{K}(s,t) = \En \bigl[ \bigl(\hat{X}_{i}(s) - \bar{\hat {X}}(s)
\bigr) \bigl(\hat{X}_{i}(t) - \bar{\hat{X}}(t) \bigr) \bigr],
\]
where $\bar{\hat{X}}(t) = n^{-1} \sum_{i=1}^{n} \hat{X}_{i}(t)$. Let
$\hat{K}(s,t)= \sum_{j=1}^{\infty} \hat{\kappa}_{j} \hat{\phi}_{j}(s)
\hat{\phi}_{j}(t)$ be the spectral expansion of $\hat{K}(s,t)$, where
$\hat{\kappa}_{1} \geq\hat{\kappa}_{2} \geq\cdots\geq0$ are ordered
eigenvalues and $\{ \hat{\phi}_{j} \}_{j=1}^{\infty}$ is an orthonormal
basis of $L_{2}[0,1]$ consisting of eigenfunctions of the integral
operator from $L_{2}[0,1]$ to itself with kernel $\hat{K}(s,t)$.
Each principal score $\xi_{ij}$ is now estimated by
\[
\hat{\xi}_{ij} = \int_{0}^{1} \bigl(
\hat{X}_{i}(t) - \bar{\hat{X}}(t) \bigr) \hat{\phi}_{j}(t)
\,dt.
\]
Then, the coefficients $a(u)$ and $b_{1}(u),\ldots,b_{m}(u)$ are
estimated by
%
%e2.4 #&#
%
\begin{equation}\quad
\bigl(\hat{a}(u),\hat{b}_{1}(u),\ldots,\hat{b}_{m}(u)
\bigr) =\mathop{\arg\min}_{a,b_{1},\ldots,b_{m}} \En \Biggl[ \rho_{u}
\Biggl(Y_{i} - a - \sum_{j=1}^{m}
b_{j} \hat{ \xi}_{ij} \Biggr) \Biggr]. \label{qrproblem}
\end{equation}
The resulting estimator for the slope function $(t,u) \mapsto b(t,u)$
is given by
\[
\hat{b}\dvtx(t,u) \mapsto\hat{b}(t,u),\qquad \hat{b}(t,u) = \sum
_{j=1}^{m} \hat{b}_{j}(u) \hat{
\phi}_{j}(t),\qquad t \in[0,1], u \in\U.
\]

The optimization problem (\ref{qrproblem}) can be transformed into a
linear programming problem and can be solved by using standard
statistical softwares.
Once the estimator for the slope function is obtained, the conditional
$u$-quantile of $Y$ given $X = x$ for a given function $x = (x(t))_{t
\in[0,1]} \in L_{2}[0,1]$ is estimated by a plug-in method:
\[
\hat{Q}_{Y | X}(u \mid x) = \hat{a}(u) + \int_{0}^{1}
\hat{b}(t,u) \bigl(x(t) - \bar{\hat{X}}(t) \bigr)\,dt.
\]
Empirical choice of the cutoff level will be discussed in Section~\ref{MC}.

The basis $\{ \phi_{j} \}_{j=1}^{\infty}$ is called the (population)
PCA basis.
Alternatively, one may use other basis functions independent of the
data, such as Fourier and wavelet bases, in which case the analysis
becomes more tractable.
A potential drawback of using such basis functions is that, as
discussed in \citet{DH11}, using the ``first'' $m$ basis functions is
less motivated.
The PCA basis is a benchmark basis in functional data analysis, which
is the reason why the PCA basis is used in this paper.
Other estimation methods such as smoothing splines [\citet{CKS09}] and a
reproducing kernel Hilbert space approach [\citet{YC10}] could be
adapted in the quantile regression case, which is left as a future topic.

The interpolation rule used here may be replaced by any other
reasonable interpolation rule.
For example, a plausible alternative is to use
\[
\hat{X}_{i}^{\mathrm{mid}}(t) = \sum_{l=1}^{L_{i}}
\frac
{X(t_{il})+X(t_{i,l+1})}{2} 1 \bigl(t \in[t_{il},t_{i,l+1})\bigr ), \qquad
i=1, \ldots,n.
\]
It is not hard to see that the theory below also applies to this
interpolation rule. In practice, this interpolation rule may be more
recommended since it uses all the discrete observations
$X_{i}(t_{i1}),\ldots,X_{i}(t_{i,L_{i}+1})$.

Finally, we note that for any fixed $u \in\U$, our estimator $\hat
{b}(\cdot,u)$ can be understood as a regularized solution to an
empirical version of a nonlinear ill-posed inverse problem that
corresponds to
the ``normal equation'' in the quantile regression case. Due to the
space limitation, we shall discuss this connection to nonlinear
ill-posed inverse problems in Appendix A in the supplementary file [\citet{supp}].

\subsection{Monotonization}
\label{monotone}

Suppose in this section that $\U$ is a bounded closed interval: $\U=[
u_{L},u_{U} ]$ with $0 < u_{L} < u_{U} < 1$. The conditional quantile
function $Q_{Y|X}(u \mid x)$ is monotonically nondecreasing in $u$.
However, the plug-in estimator $\hat{Q}_{Y|X}(u \mid x)$ constructed is
not necessarily so. To circumvent this problem, we may monotonize the
map $u \mapsto\hat{Q}_{Y|X}(u \mid x)$ by one of the following three
methods: (i) rearrangement [\citet{CFG09}], (ii) isotonization [\citet
{BBBB72}], (iii) convex combination of (i) and (ii). Such methods are
explained in \citet{CFG09} in a general setup.
We briefly explain their basic ideas. Pick any $x \in L_{2}[0,1]$.

\begin{longlist}[(iii)]
\item[(i)] Rearrangement.
% The idea of the rearrangement comes from the following simple fact:
%let $F$ be a distribution function on $\R$. Denote by $F^{-1}$ the
%quantile function of $F$. Then, we have
% \[
% F(y) = \int_{0}^{1} 1\{ F^{-1}(u) \leq y \} \,du.
% \]
% Since we only have estimates $\hat{Q}_{Y|X}(u \mid x)$ for $u \in
Define the function
\[
\hat{F}_{\U}(y \mid x) = \frac{1}{u_{U}-u_{L}} \int_{\U}
1 \bigl\{ \hat{Q}_{Y|X}(u \mid x) \leq y \bigr\} \,du.
\]
Then the map $y \mapsto\hat{F}_{\U}(y \mid x)$ is a proper
distribution function supported in $[\min_{u \in\U} \hat{Q}_{Y|X}(u
\mid x), \max_{u \in\U} \hat{Q}_{Y|X}(u \mid x)]$. The rearranged
estimator for\break $Q_{Y|X}(u \mid x)$ is defined by
\[
\hat{Q}^{*}_{Y|X}(u \mid x) = \hat{F}^{-1}_{\U}
\biggl(\frac{u -
u_{L}}{u_{U} - u_{L}} \Big| x \biggr).
\]
Clearly, the map $u \mapsto\hat{Q}_{Y|X}^{*}(u \mid x)$ is nondecreasing.

\item[(ii)] Isotonization. The isotonization is carried out by projecting the
original estimate $u \mapsto\hat{Q}_{Y|X}(u \mid x)$ on the set of
nondecreasing functions, typically by the ``pool adjacent violators''
algorithm. Denote by $\hat{Q}^{I}_{Y|X}(u \mid x)$ the isotonized estimator.

\item[(iii)] Convex combination of (i) and (ii). Take $\lambda\in[0,1]$.
Then, the convex combination $\hat{Q}^{\lambda}_{Y|X}(u \mid x) =
\lambda\hat{Q}^{*}_{Y|X}(u \mid x) + (1-\lambda) \hat{Q}_{Y|X}^{I}(u
\mid x)$ is nondecreasing in $u$.
\end{longlist}

By \citet{CFG09}, it is shown that any monotonized estimate [constructed
by using one of (i)--(iii)] is at least as good as the initial estimate
$\hat{Q}_{Y|X}(u \mid x)$ in the following sense: let $\hat
{Q}^{\dagger
}_{Y|X}(u \mid x)$ be any monotonized estimate for $Q_{Y|X}(u \mid x)$
given above.
Then, we have for all $q \in[1,\infty]$,
%
%e2.5 #&#
%
\begin{eqnarray}
\label{mono} && \biggl[ \int_{\U} \bigl| \hat{Q}^{\dagger}_{Y | X}(u
\mid x) - Q_{Y \mid
X}(u \mid x) \bigr|^{q} \,du \biggr]^{1/q}
\nonumber
\\[-8pt]
\\[-8pt]
\nonumber
&&\qquad\leq \biggl[ \int_{\U} \bigl|
\hat{Q}_{Y | X}(u \mid x) - Q_{Y \mid X}(u \mid x) \bigr|^{q}
\,du \biggr]^{1/q},
\end{eqnarray}
where the obvious modification is made when $q = \infty$.

%s3 #&#
\section{Rates of convergence}
\label{main}

In this section we derive rates of convergence for the estimators
defined in the previous section and argue their optimality.
We make the following assumptions. Let $C_{1} > 1$ be some sufficiently
large constant. First of all, we assume the following:

\begin{longlist}[(A1)]
\item[(A1)] $\{ (Y_{i},X_{i}) \}_{i=1}^{\infty}$ is i.i.d. with $(Y,X)$.
\item[(A2)] $\int_{0}^{1} \E[X^{4}(t)] \,dt \leq C_{1}$ and $\E[\xi_{j}^{4}] \leq C_{1} \kappa_{j}^{2}$ for all $j
\geq1$.
\end{longlist}

The i.i.d. assumption is conventional. It is beyond the scope of the
paper to extend the theory to dependent data. Note that \citet{HK10}
discussed weakly dependent functional data. Assumption (A2) is a mild
moment restriction. Here no moment condition on $Y$ is needed, so $\E[
Y ]$ may not exist.

\begin{longlist}[(A3)]
\item[(A3)] For some $\alpha> 1$, $C_{1}^{-1} j^{-\alpha} \leq
\kappa_{j} \leq C_{1} j^{-\alpha}$ and $\kappa_{j} - \kappa_{j+1}
\geq
C_{1}^{-1} j^{-\alpha-1}$ for all $j \geq1$.
\item[(A4)] For some $\beta> \alpha/2 + 1$, $\sup_{u \in\U} |
b_{j}(u) | \leq C_{1} j^{-\beta}$ for all $j \geq1$.
\item[(A5)] Let $F_{Y | X}(y | X)$ denote the conditional distribution
function of $Y$ given~$X$.
Then, for every realization of $X$, the map $y \mapsto F_{Y|X}(y | X)$
is twice continuously differentiable with $f_{Y | X} (y | X) = \partial
F_{Y|X}(y | X)/\partial y$ and\break
$f'_{Y|X}(y | X) = \partial f_{Y|X}(y | X)/\partial y$. Furthermore,
$f_{Y | X} (y | X) \vee| f_{Y|X}' (y | X) | \leq C_{1}$.
\item[(A6)] $\inf_{u \in\U} f_{Y|X}(Q_{Y | X}(u \mid X) | X) \geq
C_{1}^{-1}$.
\end{longlist}

Assumptions (A3) and (A4) are adapted from (3.2) and (3.3) of \citet
{HH07}. Assumption (A3) especially means that all $\kappa_{j}$ are
positive, which guarantees identification of the slope function.
In assumption (A3), $\alpha$ measures the smoothness of the covariance
kernel $K$, which also measures the difficulty of estimating the slope
function $(t,u) \mapsto b(t,u)$.
The second part of assumption (A3) is to require the spacings among
$\kappa_{j}$ not be too small, which ensures identifiability of
eigenfunctions $\phi_{j}$ and thereby sufficient estimation accuracy of
$\hat{\phi}_{j}$.
Note that the lower inequality for $\kappa_{j}$ follows essentially
from the condition on the difference, as $\kappa_{j} \to0$ ($j \to
\infty$) and hence $\kappa_{j} = \sum_{k=j}^{\infty}(\kappa_{k} -
\kappa_{k+1})$.
Assumption (A4) determines the ``smoothness'' of the function $t
\mapsto b(t,u)$.
The condition that $\beta> \alpha/2 + 1$ requires the function $t
\mapsto b(t,u)$ to be sufficiently smooth relative to $K$ uniformly in
$u \in\U$.
See \citeauthor{HH07} [(\citeyear{HH07}), page 74] for some related discussions on these assumptions.
Assumptions (A5) and (A6) are specific to the quantile regression case.
Both assumptions are standard in the quantile regression literature
when $X$ is a vector.
The role of assumption (A6) is to guarantee that the conditional
distribution function $F_{Y \mid X}(y \mid X)$ is not ``flat'' near
quantiles of interest, which is essential to our asymptotic study.

\begin{longlist}[(A7)]
\item[(A7)] For each $i=1,\ldots,n$, $X_{i}$ is observed only at
discrete points $0=t_{i1} < t_{i2} < \cdots< t_{i,L_{i}+1} = 1$.
Define $\Delta= \Delta_{n} = \max_{1 \leq i \leq n} \max_{1 \leq l
\leq L_{i}} ( t_{i,l+1} - t_{il} )$. Then, $\Delta\to0$ as $n \to
\infty$.
\item[(A8)] There exists a constant $\gamma\in(0,2]$ such that $\E[
( X(t) - X(s))^{2} ] \leq C_{1} | t-s |^{\gamma}$ for all $s,t \in[0,1]$.
\end{longlist}

Assumptions (A7) and (A8) are a set of sampling assumptions on $X_{i}$.
A~rate restriction will be imposed on $\Delta$.
Assumption (A7) at least requires that $\min_{1 \leq i \leq n} L_{i}
\to\infty$ as $n \to\infty$, which means that each set of discrete
points $t_{i1},\ldots,t_{i,L_{i}+1}$ has to be dense in $[0,1]$ as the
sample size grows. Assumption~(A8) is an additional assumption on the
smoothness of the covariance kernel as well as the mean function $t
\mapsto\E[ X(t) ]$.
Suppose in this discussion $\E[ X(t) ] = 0$ for all $t \in[0,1]$.
Then, for example, $\gamma=1$ if $K(s,t)$ is Lipschitz continuous and
$\gamma=2$ if $K(s,t)$ is twice continuously differentiable. The value
of $\gamma$ controls the discretization error.
Note that possible values of $\gamma$ depend on the value of $\alpha$.
Typically, if $\alpha$ is sufficiently large, (A8) is satisfied with
$\gamma=2$ (and vice versa). The relation between the smoothness of the
covariance (or more generally integral) kernel and the decay rate of
the associated eigenvalues is studied in \citet{RWW95} and \citet{FM09}
and the references therein. For example, the former paper shows that,
if $K$ verifies the ``Sacks--Ylvisaker condition'' of order $r$
%if essentially $K$ is $r+2$ times continuously differentiable
with $r $
a nonnegative integer (see the original paper for the precise
description of the conditions), in which case (A8) is satisfied with
$\gamma=2$, then $\kappa_{j} \asymp j^{-2r-2}$ as $j \to\infty$.
Assumption (A8) is similar in spirit to (A2) of \citet{CKS09}, in which
they directly assumed some smoothness of the random function $t \mapsto
X(t)$ to deal with the discretization error [roughly speaking, their
$2\kappa$ corresponds to our $\gamma$ as (A2) of \citet{CKS09}
essentially assumes $X$ to be $\kappa$-H\"{o}lder continuous and in
that case $\E[ (X(t) - X(s))^{2} ] =O( |t-s|^{2\kappa})$].

Let $\mathcal{F} = \mathcal{F}(C_{1},\alpha,\beta,\gamma)$ denote the
set of all distributions of $(Y,X)$ compatible with assumptions
(A2)--(A6) and (A8) for given (admissible) values of $C_{1}, \alpha,
\beta$ and $\gamma$ (such that $\mathcal{F} \neq\varnothing$). The
following theorem, which will be proved in Section~\ref{proof1} below,
establishes rates of convergence for the slope estimator $(t,u) \mapsto
\hat{b}(t,u)$.

%th3.1 #&#
%
\begin{theorem}
\label{thm1}
Suppose that assumptions \textup{(A1)--(A8)} are satisfied. Take $m =m_{n} \asymp
n^{1/(\alpha+2\beta)}$. Then, we have
%
%e3.1 #&#
%
\begin{eqnarray} \label{rate}
&&\lim_{M \to\infty} \limsup_{n \to\infty} \sup_{F \in\mathcal{F}}
\Prob_{F} \biggl[ \sup_{u \in\U} \int_{0}^{1}
\bigl\{ \hat{b}(t,u) - b(t,u) \bigr\}^{2} \,dt
\nonumber
\\[-8pt]
\\[-8pt]
\nonumber
&&\hspace*{135pt}{}> M n^{-(2\beta
-1)/(\alpha+2\beta)}
\biggr] = 0,
\end{eqnarray}
provided that $(n \vee(\log n) m^{3\alpha+ 3}) \Delta^{\gamma} =
O(1)$ as $n \to\infty$.
\end{theorem}

Inspection of the proof of Theorem~\ref{thm1} shows that, when $X$ is
continuously observable, under assumptions (A1)--(A6), the rate of
convergence of the estimator based on the direct empirical covariance
kernel will be $n^{-(2\beta-1)/(\alpha+2\beta)}$.
The side condition that $(n \vee(\log n) m^{3\alpha+ 3}) \Delta^{\gamma} = O(1)$ is assumed to make the discretization error negligible.
This condition seems not quite restrictive.
For example, if $\beta\geq\alpha+ 3/2$ and $\gamma=2$, it is
satisfied as long as $\Delta= O((n\log n)^{-1/2})$, which seems to be
mild in view of some applications in functional data analysis.

The following theorem, which will be proved in Appendix B in the
supplementary file [\citet{supp}], establishes rates of convergence for $\hat
{Q}_{Y|X}(u \mid x)$.\vadjust{\goodbreak}
For notational convenience, define
\[
\mathcal{E}( \hat{Q}_{Y|X}, u ) = \int \bigl\{ \hat{Q}_{Y | X}(u
\mid x) - Q_{Y \mid X}(u \mid x) \bigr\}^{2} \,dP_{X}(x),
\]
where $P_{X}$ denotes the distribution of $X$ (defined on $D[0,1]$).

%th3.2 #&#
%
\begin{theorem}
\label{thm2}
Suppose that assumptions \textup{(A1)--(A8)} are satisfied. Take $m =m_{n} \asymp
n^{1/(\alpha+2\beta)}$. Then, we have
%
%e3.2 #&#
%
\begin{equation}
\lim_{M \to\infty} \limsup_{n \to\infty} \sup_{F \in\mathcal{F}}
\Prob_{F} \Bigl[ \sup_{u \in\U} \mathcal{E}(\hat{Q}_{Y|X},u)
> Mn^{-(\alpha+2\beta-1)/(\alpha+2\beta)} \Bigr] = 0, \label{rate2}
\end{equation}
provided that $(n \vee(\log n) m^{3\alpha+ 3}) \Delta^{\gamma} =
O(1)$ as $n \to\infty$.
\end{theorem}

For monotonized estimators, the following corollary directly follows in
view of~(\ref{mono}) and Theorem~\ref{thm2}.

%co3.1 #&#
%
\begin{corollary}
\label{cor1}
Let $\U$ be a bounded closed interval in $(0,1)$.
Suppose that all the assumptions of Theorem~\ref{thm2} are satisfied.
Let $\hat{Q}^{\dagger}_{Y|X}(u \mid x)$ be any monotonized estimator
for $Q_{Y|X}(u \mid x)$ given in Section~\ref{monotone}.
Then, we have
\[
\lim_{M \to\infty} \limsup_{n \to\infty} \sup_{F \in\mathcal{F}}
\Prob_{F} \biggl[ \int_{\U} \mathcal{E} \bigl(
\hat{Q}^{\dagger}_{Y|X},u \bigr) \,du > Mn^{-(\alpha+2\beta
-1)/(\alpha+2\beta)} \biggr] =
0.
\]
\end{corollary}

Here note that the rate $n^{-(\alpha+2\beta-1)/(\alpha+2\beta)}$
attained in estimating $Q_{Y | X}(u \mid x)$ is faster than the rate
$n^{-(2\beta-1)/(\alpha+2\beta)}$ attained in estimating $b(t,u)$.

In what follows, we discuss optimality of these rates.

%pr3.1 #&#
%
\begin{proposition}
\label{prop1}
Suppose that assumptions \textup{(A1)--(A6)} and \textup{(A8)} are satisfied.
Let $\gamma$ be such that
%
%e3.3 #&#
%
\begin{equation}
0 < \gamma\cases{ < \alpha-1, &\quad $\mbox{if $\alpha\leq3$},$ \vspace*{2pt}
\cr
\leq2, &\quad $\mbox{if $\alpha> 3$}.$} \label{gamma}
\end{equation}
Then, there exists a constant $M > 0$ such that for $\mathcal
{F}=\mathcal{F}(C_{1},\alpha,\beta,\gamma)$,
%
%e3.4 #&#
%
\begin{eqnarray}\label{lower1}
&&\liminf_{n \to\infty} \inf_{\bar{b}} \sup_{F \in\mathcal{F}}
\Prob_{F} \biggl[ \sup_{u \in\U} \int_{0}^{1}
\bigl\{ \bar{b}(t,u) - b(t,u) \bigr\}^{2} \,dt
\nonumber
\\[-8pt]
\\[-8pt]
\nonumber
&&\hspace*{118pt}{}> M n^{-(2\beta
-1)/(\alpha+2\beta)}
\biggr] > 0,
\end{eqnarray}
where $\inf_{\bar{b}}$ is taken over all estimators for the slope
function $(t,u) \mapsto b(t,u)$ based on $(Y_{1},X_{1}),\ldots
,(Y_{n},X_{n})$. Similarly, there exists a constant $M > 0$ such that
for $\mathcal{F}=\mathcal{F}(C_{1},\alpha,\beta,\gamma)$,
\[
\liminf_{n \to\infty} \inf_{\bar{Q}_{Y|X}} \sup_{F \in\mathcal{F}}
\Prob_{F} \Bigl[ \sup_{u \in\U} \mathcal{E}(\bar{Q}_{Y|X},u)
> M n^{-(\alpha+ 2\beta-1)/(\alpha+2\beta)} \Bigr] > 0,
\]
where $\inf_{\bar{Q}_{Y|X}}$ is taken over all estimators for the
conditional quantile function $Q_{Y|X}\dvtx (u,x) \mapsto Q_{Y|X}(u
\mid
x)$ based on $(Y_{1},X_{1}),\ldots,(Y_{n},X_{n})$. When $\U$ is a
bounded closed interval in $(0,1)$, there exists a constant $M > 0$
such that for $\mathcal{F}=\mathcal{F}(C_{1},\alpha,\beta,\gamma)$,
\[
\liminf_{n \to\infty} \inf_{\bar{Q}_{Y|X}} \sup_{F \in\mathcal{F}}
\Prob_{F} \biggl[ \int_{\U} \mathcal{E}(
\bar{Q}_{Y|X},u)\,du > M n^{-(\alpha+ 2\beta-1)/(\alpha+2\beta)} \biggr] > 0,
\]
where the previous convention applies.
\end{proposition}

The side condition (\ref{gamma}) is a compatibility condition between
assumptions~(A3) and (A8). It is not addressed here whether this
condition is tight.
However, some restriction between $\alpha$ and $\gamma$ is required in
establishing lower bounds of rates of convergence to guarantee that the class
$\mathcal{F}(C_{1},\alpha,\beta,\gamma)$ is at least nonempty.
Proposition~\ref{prop1} shows that under this side condition the rates
established in Theorems~\ref{thm1},~\ref{thm2} and Corollary~\ref{cor1}
are indeed optimal in the minimax sense. A proof of Proposition \ref
{prop1} is given in Appendix C in the supplementary file [\citet{supp}].

% This proposition shows that the slope estimator $(u,t) \mapsto
%satisfying (\ref{gamma}).

%s4 #&#
\section{Empirical choice of the cutoff level}
\label{MC}

In this section we suggest three criteria to choose $m$, and
investigate their performance by simulations.\setcounter{footnote}{1}\footnote{Application of a
(weighted) Lasso type procedure could be an alternative in the $m$-selection.
Also, a thresholding rule for the estimated eigenvalues is a possible
option. Analysis of such procedures is left to a future work.}
We use a heuristic reasoning to derive selection criteria. Suppose that
$\U$ is a singleton: $\U= \{ u \}$. Suppose that there is no
truncation bias, that is, $b(t,u)= \sum_{j=1}^{m} b_{j} (u)\phi_{j}(t)$
and $Q_{Y|X}(u | X) = a(u) + \sum_{j=1}^{m} b_{j}(u) \xi_{j}$. Then,
the infeasible estimator $(\tilde{a}(u),\tilde{b}_{1}(u),\ldots
,\tilde
{b}_{m}(u))'$ defined by (\ref{inf}) can be regarded as a (conditional)
maximum likelihood estimator when the conditional distribution of $Y$
given~$X$ has the asymmetric Laplace density of the form
\[
f(y | X, u, \sigma) = \frac{u(1-u)}{\sigma} \exp \Biggl\{ - \frac
{1}{\sigma}
\rho_{u} \Biggl(y-a(u)- \sum_{j=1}^{m}
b_{j}(u) \xi_{j} \Biggr) \Biggr\},
\]
where $\sigma> 0$ is a scale parameter.
This suggests the following analogues of AIC and BIC in the present context:
\begin{eqnarray*}
\AIC(u) &=& \log \Biggl[ \frac{1}{n} \sum_{i=1}^{n}
\rho_{u} \Biggl(Y_{i}-\hat{a}(u)- \sum
_{j=1}^{m} \hat{b}_{j}(u) \hat{
\xi}_{ij} \Biggr) \Biggr] + \frac{(m+1)}{n},
\\
\BIC(u) &=& \log \Biggl[ \frac{1}{n} \sum_{i=1}^{n}
\rho_{u} \Biggl(Y_{i}-\hat{a}(u)- \sum
_{j=1}^{m} \hat{b}_{j}(u) \hat{
\xi}_{ij} \Biggr) \Biggr] + \frac{(m+1) \log n}{n}.
\end{eqnarray*}
See also \citeauthor{K05} [(\citeyear{K05}), Section 4.9.1] for some related discussion.
According to \citet{Y06}, we may also consider an analogue of GACV as follows:
\[
\GACV(u) = \frac{\sum_{i=1}^{n} \rho_{u}(Y_{i}-\hat{a}(u)-
\sum_{j=1}^{m} \hat{b}_{j}(u) \hat{\xi}_{ij})}{n-(m+1)}.
\]

When $\U$ is a bounded closed interval, define the integrated-AIC, BIC
and GACV as follows:
\begin{eqnarray*}
\IAIC&=& \int_{\U} \AIC(u) \,du,\qquad \IBIC= \int
_{\U} \BIC(u) \,du,\\
 \IGACV&=& \int_{\U}
\GACV(u) \,du.
\end{eqnarray*}
When $\U$ is a set of finite grid points, each integral is replaced by
the summation over the grid points.

We carried out a small Monte Carlo study to investigate the finite
sample performance of these criteria.
In all cases, the number of Monte Carlo repetitions was $1000$.
The numerical results obtained in this section were carried out by
using the matrix language Ox [\citet{D02}]. The Ox code for solving
quantile regression problems supplied on Professor Koenker's website
was used. See also \citet{KP97} for some computational aspects of
quantile regression problems.

The simulation design is described as follows:
\begin{eqnarray}
Y &= &\int_{0}^{1} \varrho(t) X(t) \,dt +
\varepsilon,
\nonumber\\
\varrho(t)&= &\sum_{j=1}^{50}
\varrho_{j} \phi_{j}(t),\nonumber \\
\eqntext{\varrho_{1} = 0.3,
\varrho_{j} = 4 (-1)^{j+1} j^{-2}, j \geq2,
\phi_{j}(t) = 2^{1/2} \cos(j \pi t),}
\\
X(t)& =& \sum_{j=1}^{50}
\gamma_{j} Z_{j} \phi_{j}(t), \nonumber\\
\eqntext{\gamma_{j} = (-1)^{j+1} j^{-\alpha/2}, \alpha\in\{ 1.1, 2
\}, Z_{j} \sim U \bigl[-3^{1/2},3^{1/2} \bigr],}
\\
\varepsilon&\sim &N(0,1) \mbox{ or } \mbox{Cauchy},\qquad n \in\{ 100, 200, 500
\}.\nonumber
\end{eqnarray}
This $\alpha$ corresponds to $\alpha$ in assumption (A3). Each $X_{i}$
is observed at $201$ equally spaced grid points on $[0,1]$.
In this design, we have
\[
Q_{Y|X}(u \mid X) = F_{\varepsilon}^{-1}(u) + \int
_{0}^{1} \varrho(t) X(t) \,dt,
\]
where $F_{\varepsilon}^{-1}(\cdot)$ is the quantile function of the
distribution of $\varepsilon$. Thus, $a(u) = F_{\varepsilon}^{-1}(u)$
and $b(t,u) \equiv\varrho(t)$ [$b(t,u)$ is independent of $u$].
We considered two cases for $\U$: (a) $\U= \{ 0.5 \}$ and (b) $\U= \{
0.15,0.2,\ldots,0.85 \}$. In each case, the performance was measured by
\begin{eqnarray*}
\QAMISE&=& \frac{1}{\Card(\U)}\sum_{u \in\U} \E \biggl[
\int_{0}^{1} \bigl\{ \hat{b}(t,u) - b(t,u) \bigr
\}^{2} \,dt \biggr] \quad \mbox{or}
\\
\QAMISE&=& \frac{1}{\Card(\U)}\sum_{u \in\U} \E \biggl[
\int \bigl\{ \hat{Q}_{Y|X}(u \mid x) - Q_{Y|X}(u \mid x) \bigr
\}^{2} \,dP_{X}(x) \biggr],
\end{eqnarray*}
where $P_{X}$ denotes the distribution of $X$ and $\QAMISE$ is the
abbreviation of ``quantile-averaged mean integrated squared error.''

The simulation results for case (a) are summarized in Figures \ref
{fig1}--\ref{fig4} and Table~\ref{table1}.
Figures~\ref{fig1} and~\ref{fig2} show the performance of the selection
criteria for the normal error case, while Figures~\ref{fig3} and \ref
{fig4} show that for the Cauchy error case.
In each figure, ``Fixed'' refers to the performance of the estimator
with fixed $m$. Table~\ref{table1} shows the average numbers of $m$
selected by three criteria. The table shows that as the sample size
increases, all the criteria tend to choose larger $m$, and as $\alpha$
increases, they tend to choose smaller $m$, which is consistent with
the theoretical requirement in the $m$-selection. Interestingly, it is
observed that the error distribution affects the performance of the
$m$-selection.

%f1 #&#
%
\begin{figure}

\includegraphics{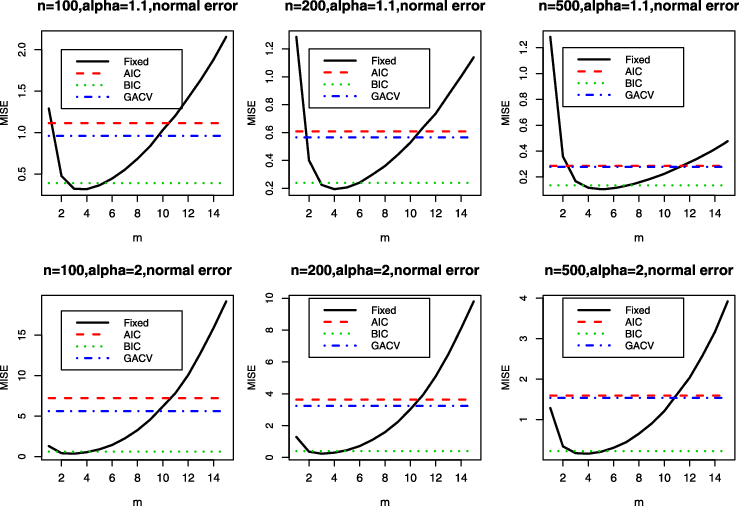}

\caption{Performance of selection criteria. Case \textup{(a)}.
Estimation of $b(\cdot,0.5)$.} \label{fig1}
\end{figure}

%f2 #&#
%
\begin{figure}

\includegraphics{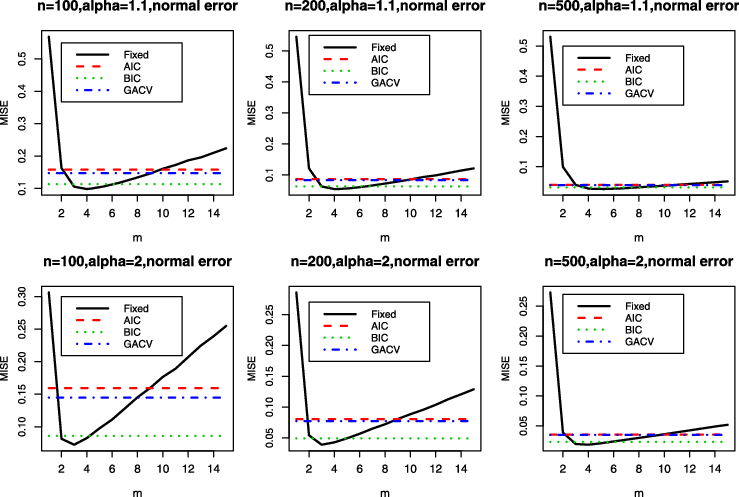}

\caption{Performance of selection criteria. Case \textup{(a)}.
Estimation of $Q_{Y|X}(0.5 \mid x)$. } \label{fig2}
\end{figure}

%f3 #&#
%
\begin{figure}

\includegraphics{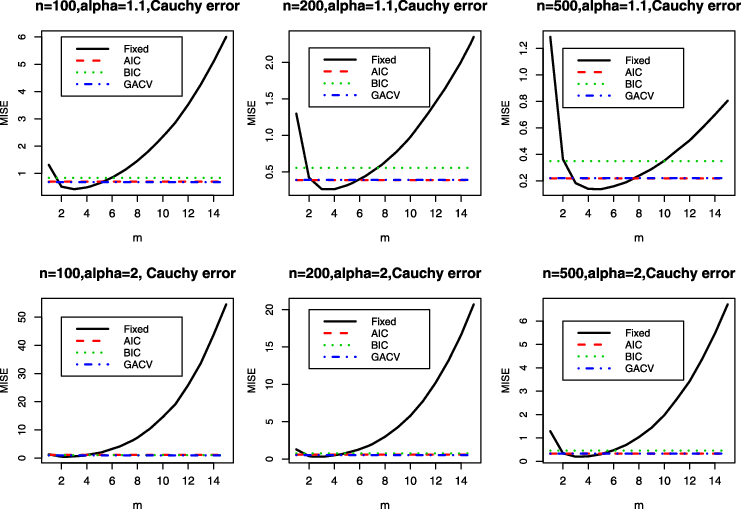}

\caption{Performance of selection criteria. Case \textup{(a)}.
Estimation of $b(\cdot,0.5)$. } \label{fig3}
\end{figure}

%f4 #&#
%
\begin{figure}

\includegraphics{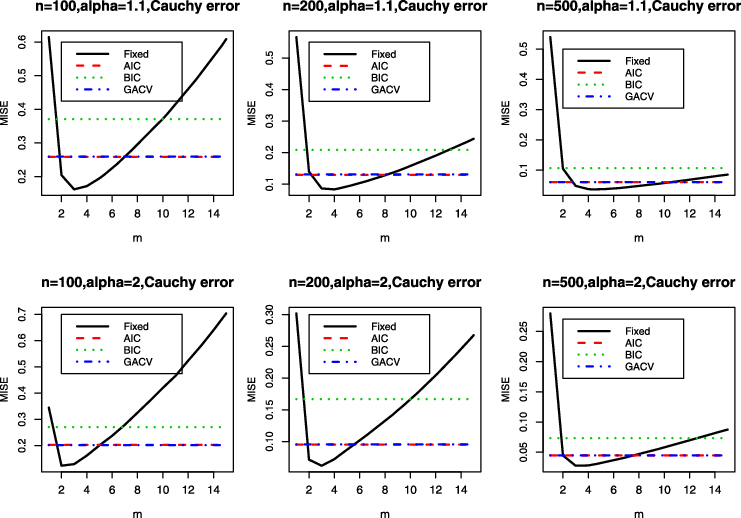}

\caption{Performance of selection criteria. Case \textup{(a)}.
Estimation of $Q_{Y|X}(0.5 \mid x)$.} \label{fig4}
\end{figure}

In the normal error case, in Figure~\ref{fig1}, BIC worked better than
the other two criteria. On the other hand, in the Cauchy error case, in
Figure~\ref{fig2}, AIC and GACV worked better than BIC. Looking closely
at these figures, one finds that AIC and GACV performed quite badly in
some cases (see the bottom half in Figure~\ref{fig1}). Although none of
these criteria clearly dominated the others, BIC worked relatively stably.
These figures also show that as $\alpha$ increases from $1.1$ to $2$,
the performance of $\hat{b}(\cdot,0.5)$ becomes worse, while that of
$\hat{Q}_{Y|X}(0.5 \mid x)$ becomes better. This is consistent with the
theoretical results in the previous section. Essentially similar
comments apply to case (b), Figures 1--4 and Table 1 in Appendix F in
the supplementary file [\citet{supp}].

%t1 #&#
%
\begin{table}
\caption{Average numbers of $m$ selected by three criteria in case \textup{(a)}.
Standard deviations are given in parentheses}\label{table1}
\begin{tabular*}{\textwidth}{@{\extracolsep{\fill}}lccccc@{}}
\hline
\multicolumn{1}{@{}l}{$\bolds{n}$} &
\multicolumn{1}{c}{$\bolds{\alpha}$} & \multicolumn{1}{c}{\textbf{error
dist.}} &
\multicolumn{1}{c}{\textbf{AIC}} & \multicolumn{1}{c}{\textbf{BIC}} &
\multicolumn{1}{c@{}}{\textbf{GACV}} \\
\hline
100 & 1.1 & normal & 7.40 (4.12) & 3.09 (1.24)& 6.58 (3.78)\\
200& 1.1 & normal & 7.34 (3.87) & 3.39 (1.14)& 6.98 (3.67)\\
500 & 1.1 & normal & 8.13 (3.59)& 3.94 (1.05)& 7.97 (3.52)\\
100 & 2\phantom{0.} & normal & 6.50 (4.23)& 2.53 (0.93)& 5.67 (3.84)\\
200 & 2\phantom{0.} & normal & 6.33 (3.95) & 2.67 (0.85) & 5.96 (3.76)
\\
500 & 2\phantom{0.} & normal & 6.92 (3.85) & 3.12 (0.77) & 6.81 (3.80)
\\
100 & 1.1 & Cauchy & 2.30 (1.18) & 1.62 (0.61) & 2.24 (1.08) \\
200 & 1.1 & Cauchy & 2.51 (0.88) & 1.91 (0.64) & 2.48 (0.84) \\
500 & 1.1 & Cauchy & 3.00 (0.87) & 2.20 (0.51) & 2.99 (0.86) \\
100 & 2\phantom{0.} & Cauchy & 1.90 (0.94) & 1.35 (0.51) & 1.86 (0.81)
\\
200 & 2\phantom{0.} & Cauchy & 2.14 (0.76) & 1.60 (0.53) & 2.10 (0.60)
\\
500 & 2\phantom{0.} & Cauchy & 2.41 (0.67) & 1.91 (0.38) & 2.40 (0.67)
\\
\hline
\end{tabular*}
\end{table}

%s5 #&#
\section{\texorpdfstring{Proof of Theorem \protect\ref{thm1}}{Proof of Theorem 3.1}}
\label{proof1}

We divide the proof into three subsections. Some technical results are
proved in Appendices D and E in the supplementary file [\citet{supp}].
To avoid the notational complication, uniformity in $F \in\mathcal{F}$
will be suppressed.
Let $C > 0$ denote a generic constant of which the value may change
from line to line.
In most cases, qualification ``almost surely'' will be suppressed.
In some parts of the proofs, we use empirical process techniques.
We follow the basic notation used in \citet{VW96}.

%s5.1 #&#
\subsection{Reduction of the problem}\label{sec5.1}

Let $b_{0}(u) = a(u), \hat{b}_{0}(u) = \hat{a}(u)$ and $\xi_{i0} =
\hat
{\xi}_{i0} = 1$.
For any $v_{0},\ldots,v_{m} \in\R$, write $v^{m} = (v_{0},\ldots,v_{m})'$.
For $v^{m} \in\mathbb{R}^{m+1}$ and $w^{m} \in\mathbb{R}^{m+1}$,
write $v^{m} \cdot w^{m} = \sum_{j=0}^{m} v_{j} w_{j}$. Then,
\[
\hat{b}^{m}(u) = \bigl(\hat{b}_{0}(u),
\hat{b}_{1}(u),\ldots,\hat{b}_{m}(u) \bigr)' =
\mathop{\arg\min}_{b^{m} \in\mathbb{R}^{m+1}} \En \bigl[ \rho_{u} \bigl(Y_{i} -
\hat{\xi}_{i}^{m} \cdot b^{m} \bigr) \bigr].
\]
We use a further re-parameterization. For $i=1,\ldots,n$ and $j \geq1$,
let $\eta_{ij} = \kappa_{j}^{-1/2} \xi_{ij}, \hat{\eta}_{ij} =
\kappa_{j}^{-1/2} \hat{\xi}_{ij}, d_{j} (u) = \kappa_{j}^{1/2}
b_{j}(u)$ and
$\hat{d}_{j} (u)= \kappa_{j}^{1/2} \hat{b}_{j}(u)$.
For $j = 0$, define $\eta_{i0} = \hat{\eta}_{i0} = 1$, $d_{0}(u) =
b_{0}(u)$ and $\hat{d}_{0}(u) = \hat{b}_{0}(u)$. Note that $\E[ \eta_{ij} ] = 0, \E[ \eta_{ij}^{2} ] = 1$ for all $j \geq1$, and $\E[
\eta_{ij} \eta_{ik} ] = 0$ for all $j,k \geq0$ with $j \neq k$. Then,
\[
\hat{d}^{m}(u) = \bigl(\hat{d}_{0}(u),
\hat{d}_{1}(u),\ldots,\hat{d}_{m}(u) \bigr)'=
\mathop{\arg \min}_{d^{m} \in\mathbb{R}^{m+1}} \En \bigl[ \rho_{u}
\bigl(Y_{i} - \hat{ \eta}_{i}^{m} \cdot
d^{m} \bigr) \bigr].
\]
We first consider bounding $\sup_{u \in\U} \| \hat{d}^{m}(u) -
d^{m}(u) \|_{\ell^{2}}$.

%le5.1 #&#
%
\begin{lemma}
\label{lemred}
Suppose that for all $\epsilon> 0$, there exist constants $c > 0$ and
$M > 0$ possibly depending on $\epsilon$ such that
\begin{eqnarray*}
&&\liminf_{n \to\infty} \Prob \bigl\{ - \En \bigl[ \bigl\{ u-1
\bigl(Y_{i} \leq\hat{\eta}_{i}^{m} \cdot
\bigl(d^{m}(u) + M \sqrt{m/n} h^{m} \bigr) \bigr) \bigr\}
\bigl(h^{m} \cdot\hat{\eta}_{i}^{m} \bigr) \bigr]
\\
&&\hspace*{147pt}\qquad > c \sqrt{m/n}\ \forall u \in\U, \forall h^{m} \in
\mathbb{S}^{m} \bigr\} > 1-\epsilon.
\end{eqnarray*}
Then, we have
\[
\limsup_{n \to\infty} \mathbb{P} \Bigl\{ \sup_{u \in\U} \bigl\| \hat
{d}^{m}(u) - d^{m}(u) \bigr\|_{\ell^{2}} > M \sqrt{m/n}
\Bigr\} \leq\epsilon.
\]
\end{lemma}

Hence, as soon as the hypothesis of Lemma~\ref{lemred} is verified,
we will have $\sup_{u \in\U} \| \hat{d}^{m}(u) - d^{m}(u) \|_{\ell
^{2}} = O_{P}(\sqrt{m/n})$, from which the desired conclusion follows
in a relatively straightforward manner.
We will verify the hypothesis of Lemma~\ref{lemred} in Section~\ref{sec5.2}
and complete the proof of Theorem~\ref{thm1} in Section~\ref{sec5.3}.

\begin{pf*}{Proof of Lemma~\ref{lemred}}
The proof is divided into three steps.

\textit{Step} 1: $\| \En[ \{ u-1(Y_{i} \leq\hat{\eta}_{i}^{m} \cdot\hat
{d}^{m}(u)) \} \hat{\eta}_{i}^{m}] \|_{\ell^{2}} \leq((m+1)/n) \max_{1
\leq i \leq n} \| \hat{\eta}_{i}^{m} \|_{\ell^{2}}$.

The proof is based on the next lemma.
%
%le5.2 #&#
%
\begin{lemma}
\label{lemGJ}
Let $\{ (y_{i},z_{i}')' \}_{i=1}^{n}$ be a sequence of pairs of
nonstochastic variables $(y_{i},z_{i}')'$ where $y_{i} \in\mathbb{R}$
and $z_{i} \in\mathbb{R}^{k}$.
Pick any $u \in(0,1)$. Let $\check{d}(u) \in\mathbb{R}^{k}$ be any
solution to the minimization problem
\[
\min_{d \in\mathbb{R}^{k}}\En \bigl[ \rho_{u} \bigl(y_{i} -
z_{i}'d \bigr) \bigr].
\]
Then, we have
\begin{eqnarray*}
&& \bigl\| \En \bigl[ \bigl\{ u-1 \bigl(y_{i} \leq z_{i}'
\check{d}(u) \bigr) \bigr\} z_{i} \bigr] \bigr\|_{\ell^{2}}
\\
&&\qquad\leq n^{-1} \Card \bigl( \bigl\{ i \in\{ 1,\ldots,n \} \dvtx
y_{i} = z_{i}'\check{d}(u) \bigr\} \bigr)
\max_{1 \leq i \leq n} \| z_{i} \|_{\ell^{2}}.
\end{eqnarray*}
\end{lemma}

\begin{pf}
The proof follows from a small modification of El-Attar, Vidyasa\-gar and Dutta [(\citeyear{EVD79}), Lemma 2.1].
\end{pf}

Recall that $\hat{\eta}_{i}^{m}$ depends only on $X_{1}^{n} := \{
X_{1},\ldots,X_{n} \}$ and not on $Y_{1},\ldots,Y_{n}$.
Since the conditional distribution of $Y_{1},\ldots,Y_{n}$ given
$X_{1}^{n}$ is absolutely continuous,
by Sard's theorem [see \citet{M65}, Section 3], $\sup_{u \in\U}
\Card
(\{ i \in\{ 1, \ldots, n \} \dvtx Y_{i} = \hat{\eta}_{i}^{m} \cdot
\hat
{d}^{m}(u) \}) \leq m+1$ almost surely. To be more precise, pick any
subset $I \subset\{ 1,\ldots, n \}$ such that $\Card(I) \geq m+2$.
Conditional on $X_{1}^{n}$, consider the set
\[
S_{I} = \bigl\{ \bigl(\hat{\eta}_{i}^{m} \cdot
\delta^{m} \bigr)_{i \in I} \dvtx\delta^{m} \in
\mathbb{R}^{m+1} \bigr\} \subset\mathbb{R}^{\Card(I)}.
\]
Then, $S_{I}$ is a linear subspace of dimension at most $m+1$. Suppose
that $Y_{i} = \hat{\eta}_{i}^{m} \cdot\hat{d}^{m}(u)$ for all $i
\in
I$ for some $u \in\U$.
Then, $(Y_{i})_{i \in I} \in S_{I}$, by which we have
%
%e5.1 #&#
%
\begin{equation}
\Prob \bigl\{ Y_{i} = \hat{\eta}_{i}^{m} \cdot
\hat{d}^{m}(u),\ \forall i \in I, \exists u \in\U\mid
X_{1}^{n} \bigr\} \leq\Prob \bigl\{ (Y_{i})_{i \in I}
\in S_{I} \mid X_{1}^{n} \bigr\}. \label{sard}
\end{equation}
However, by Sard's theorem, the Lebesgue measure of $S_{I}$ in $\mathbb
{R}^{\Card(I)}$ is zero, and by the absolute continuity of the
conditional distribution of $(Y_{i})_{i \in I}$ given $X_{1}^{n}$, the
right side of (\ref{sard}) is zero. Thus, we conclude that
\begin{eqnarray*}
&&\Prob \Bigl\{ \sup_{u \in\U} \Card \bigl( \bigl\{ i \in\{ 1, \ldots, n \}
\dvtx Y_{i} = \hat{\eta}_{i}^{m} \cdot
\hat{d}^{m}(u) \bigr\} \bigr) \geq m+2 \mid X_{1}^{n}
\Bigr\}
\\
&&\qquad\leq\mathop{\sum_{I \subset\{ 1,\ldots, n \}}}_{\Card(I) \geq m+2
} \Prob \bigl\{
Y_{i} = \hat{\eta}_{i}^{m} \cdot
\hat{d}^{m}(u)\ \forall i \in I, \exists u \in\U\mid
X_{1}^{n} \bigr\} = 0,
\end{eqnarray*}
by which we have $\sup_{u \in\U} \Card(\{ i \in\{ 1, \ldots, n \}
\dvtx Y_{i} = \hat{\eta}_{i}^{m} \cdot\hat{d}^{m}(u) \}) \leq m+1$
almost surely.
Then, the conclusion of Step 1 follows from an application of Lemma
\ref
{lemGJ}.

\textit{Step} 2: We have
%
%e5.2 #&#
%
\begin{equation}
\max_{1 \leq i \leq n} \bigl\| \hat{\eta}_{i}^{m}
\bigr\|_{\ell^{2}} = o_{P} \bigl\{ (\log n)^{-1} \sqrt{n/m}
\bigr\}. \label{max}
\end{equation}
We defer the proof of (\ref{max}) to Appendix D.

\textit{Step} 3: Proof of the lemma.

Define
\[
\hat{h}^{m}(u) = \cases{\displaystyle \frac{\hat{d}^{m}(u) - d^{m}(u)}{\| \hat
{d}^{m}(u) - d^{m}(u) \|_{\ell
^{2}}}, &\quad $
\mbox{if $ \hat{d}^{m}(u) \neq d^{m}(u)$},$ \vspace*{2pt}
\cr
0, & \quad $ \mbox{otherwise}.$}
\]
Then, by Steps 1 and 2, we have
\[
\sup_{u \in\U} \bigl| \En \bigl[ \bigl\{ u-1 \bigl(Y_{i} \leq\hat{
\eta}_{i}^{m} \cdot\hat{d}^{m}(u) \bigr) \bigr\}
\bigl(\hat{h}^{m}(u) \cdot\hat{\eta}_{i}^{m}
\bigr) \bigr] \bigr| = o_{P} (\sqrt{m/n}).
\]
Define the event
\begin{eqnarray*}
&&\mathcal{E}_{n} := \bigl\{ - \En \bigl[ \bigl\{ u-1
\bigl(Y_{i} \leq\hat{\eta}_{i}^{m} \cdot
\bigl(d^{m}(u) + M \sqrt{m/n} h^{m} \bigr) \bigr) \bigr\}
\bigl(h^{m} \cdot\hat{\eta}_{i}^{m} \bigr) \bigr]
\\
&&\hspace*{137pt}\qquad > c \sqrt{m/n}\ \forall u \in\U, \forall h^{m} \in
\mathbb{S}^{m} \bigr\}.
\end{eqnarray*}
Since the map
\[
l \mapsto- \En \bigl[ \bigl\{ u-1 \bigl(Y_{i} \leq\hat{
\eta}_{i}^{m} \cdot \bigl(d^{m}(u) + l \sqrt{m/n}
h^{m} \bigr) \bigr) \bigr\} \bigl(h^{m} \cdot\hat{
\eta}_{i}^{m} \bigr) \bigr]
\]
is nondecreasing for all $u \in\U$ and $h^{m} \in\mathbb{S}^{m}$, the
event $\mathcal{E}_{n}$ is also written as
\begin{eqnarray*}
&&\mathcal{E}_{n} = \bigl\{ - \En \bigl[ \bigl\{ u-1
\bigl(Y_{i} \leq\hat{\eta}_{i}^{m} \cdot
\bigl(d^{m}(u) + l \sqrt{m/n} h^{m} \bigr) \bigr) \bigr\}
\bigl(h^{m} \cdot\hat{ \eta}_{i}^{m} \bigr)
\bigr]
\\
&&\hspace*{88pt}\qquad > c \sqrt{m/n}\ \forall u \in\U, \forall l \geq M, \forall
h^{m} \in \mathbb{S}^{m} \bigr\}.
\end{eqnarray*}
Thus, as $n \to\infty$,
\begin{eqnarray*}
&&\mathbb{P} \Bigl\{ \sup_{u \in\U} \bigl\| \hat{d}^{m}(u) -
d^{m}(u) \bigr\|_{\ell^{2}} > M\sqrt{m/n} \Bigr\}
\\
&&\qquad=\mathbb{P} \bigl\{ \bigl\| \hat{d}^{m}(u) - d^{m}(u)
\bigr\|_{\ell^{2}} > M\sqrt{m/n}, \exists u \in\U \bigr\}
\\
&&\qquad\leq\mathbb{P} \bigl[ \bigl\{ \bigl\| \hat{d}^{m}(u) -
d^{m}(u) \bigr\|_{\ell^{2}} > M\sqrt{m/n}, \exists u \in\U \bigr\} \cap
\mathcal{E}_{n} \bigr] +\mathbb{P} \bigl(\mathcal{E}_{n}^{c}
\bigr)
\\
&&\qquad\leq\mathbb{P} \bigl\{ - \En \bigl[ \bigl\{ u-1 \bigl(Y_{i}
\leq \hat{ \eta}_{i}^{m} \cdot\hat{d}^{m}(u)
\bigr) \bigr\} \bigl(\hat{h}^{m}(u) \cdot\hat{\eta}_{i}^{m}
\bigr) \bigr] > c \sqrt{m/n}, \exists u \in\U \bigr\}
\\
&&\qquad\quad{} + \mathbb{P} \bigl(\mathcal{E}_{n}^{c}
\bigr)
\\
&&\qquad\leq o(1) + \bigl(1+o(1) \bigr) \epsilon.
\end{eqnarray*}
This completes the proof of Lemma~\ref{lemred}.
\end{pf*}

%s5.2 #&#
\subsection{\texorpdfstring{Verification of the hypothesis of Lemma \protect\ref{lemred}}
{Verification of the hypothesis of Lemma 5.1}}\label{sec5.2}

% We now verify the hypothesis of Lemma~\ref{lemred}.
Pick any $h^{m} = (h_{0},\break h_{1},\ldots,  h_{m})' \in\mathbb{S}^{m}$. For a
given $M > 1$, let $\delta^{m} = (\delta_{0},\delta_{1},\ldots
,\delta_{m})' = M\sqrt{m/n} h^{m}$.
Then,
\begin{eqnarray*}
&&- \En \bigl[ \bigl\{ u-1 \bigl(Y_{i} \leq\hat{\eta}_{i}^{m}
\cdot \bigl(d^{m}(u) + \delta^{m} \bigr) \bigr) \bigr\}
\bigl(h^{m} \cdot\hat{\eta}_{i}^{m} \bigr) \bigr]
\\
&&\quad=- \En \bigl[ \bigl\{ u-1 \bigl(Y_{i} \leq Q_{Y|X}(u
\mid X_{i}) \bigr) \bigr\} \bigl(h^{m} \cdot\hat{
\eta}_{i}^{m} \bigr) \bigr]
\\
&&\qquad{} + \En \bigl[ \bigl\{ F_{Y|X} \bigl( \hat{
\eta}_{i}^{m} \cdot \bigl(d^{m}(u) +
\delta^{m} \bigr) \mid X_{i} \bigr) - F_{Y|X}
\bigl(Q_{Y|X}(u \mid X_{i}) \mid X_{i} \bigr) \bigr
\} \bigl(h^{m} \cdot\hat{\eta}_{i}^{m} \bigr)
\bigr]
\\
&&\qquad{} + n^{-1/2} \mathbb{G}_{n | X} \bigl[ \bigl\{ 1
\bigl(Y_{i} \leq\hat{\eta}_{i}^{m} \cdot
\bigl(d^{m}(u) + \delta^{m} \bigr) \bigr) - 1
\bigl(Y_{i} \leq Q_{Y|X}(u \mid X_{i}) \bigr) \bigr
\}\\
&&\hspace*{273pt}{}\times \bigl(h^{m} \cdot\hat{\eta}_{i}^{m} \bigr)
\bigr]
\\
&&\quad=: I + \mathit{II} + \mathit{III},
\end{eqnarray*}
where we have used the fact that $F_{Y|X} (Q_{Y|X}(u \mid X) \mid X) =
u$ and the last term (\textit{III}) is defined by
\begin{eqnarray*}
&&\hspace*{-4pt}n^{-1/2} \mathbb{G}_{n | X} \bigl[ \bigl\{ 1
\bigl(Y_{i} \leq\hat{\eta}_{i}^{m} \cdot
\bigl(d^{m}(u) + \delta^{m} \bigr) \bigr) - 1
\bigl(Y_{i} \leq Q_{Y|X}(u \mid X_{i}) \bigr) \bigr
\} \bigl(h^{m} \cdot\hat{\eta}_{i}^{m} \bigr)
\bigr]
\\
&&\hspace*{-4pt}\quad:= \En \bigl[ \bigl\{ 1 \bigl(Y_{i} \leq\hat{
\eta}_{i}^{m} \cdot \bigl(d^{m}(u) +
\delta^{m} \bigr) \bigr) - 1 \bigl(Y_{i} \leq
Q_{Y|X}(u \mid X_{i}) \bigr)
\\
&&\hspace*{16pt}\qquad{} - F_{Y|X} \bigl( \hat{\eta}_{i}^{m}
\cdot \bigl(d^{m}(u) + \delta^{m} \bigr) \mid
X_{i} \bigr) + F_{Y|X} \bigl(Q_{Y|X}(u \mid
X_{i}) \mid X_{i} \bigr) \bigr\} \bigl(h^{m}
\cdot\hat{ \eta}_{i}^{m} \bigr) \bigr].
\end{eqnarray*}
We separately bound the terms $I,\mathit{II}$ and $\mathit{III}$ uniformly in $u \in\U$
and $h^{m} \in\mathbb{S}^{m}$. In what follows, stochastic orders are
interpreted independent of $M$.
Note that
%
%e5.3 #&#
%
\begin{eqnarray}
\label{decomp} &&\hat{\eta}_{i}^{m} \cdot
\bigl(d^{m}(u) + \delta^{m} \bigr)- Q_{Y|X} (u \mid
X_{i})
\nonumber
\\
&&\qquad= \sum_{j=0}^{m}
\bigl(d_{j}(u) + \delta_{j} \bigr) \hat{\eta}_{ij}
- \sum_{j=0}^{\infty} \,d_{j}(u)
\eta_{ij}
\\
&&\qquad= \hat{\eta}_{i}^{m} \cdot
\delta^{m} + \bigl( \hat{\eta}^{m}_{i} -
\eta_{i}^{m} \bigr) \cdot d^{m}(u) - \sum
_{j=m+1}^{\infty} d_{j}(u)
\eta_{ij}\nonumber
\\
&&\qquad=: \hat{\eta}_{i}^{m} \cdot\delta^{m} +
\hat{r}_{i}(u).\nonumber
\end{eqnarray}

Bounding $I$: observe that
\[
I \geq- \bigl\| \En \bigl[ \bigl\{ u-1\bigl(Y_{i} \leq Q_{Y|X}(u \mid
X_{i})\bigr) \bigr\} \hat{\eta}_{i}^{m} \bigr]
\bigr\|_{\ell^{2}}.
\]
Using the relation
\begin{eqnarray*}
1 \bigl(Y_{i} \leq Q_{Y|X}(u \mid X_{i}) \bigr)
&=& 1 \bigl(F_{Y|X}(Y_{i} | X_{i}) \leq u \bigr)
\\
& =& 1(U_{i} \leq u) \qquad \mbox{with } U_{i} =
F_{Y|X}(Y_{i} | X_{i}),
\end{eqnarray*}
we have
\begin{eqnarray*}
\sup_{u \in\U} \bigl\| \En \bigl[ \bigl\{ u - 1 \bigl(Y_{i} \leq
Q_{Y|X}(u \mid X_{i}) \bigr) \bigr\} \hat{
\eta}_{i}^{m} \bigr] \bigr\|_{\ell^{2}}
\leq\sqrt{ \sum_{j=0}^{m}
\sup_{u \in\U} \bigl( \En \bigl[ \bigl\{ u-1(U_{i} \leq u) \bigr
\} \hat{\eta}_{ij} \bigr] \bigr)^{2} }.
\end{eqnarray*}
Here $U_{1},\ldots,U_{n}$ are independent uniform random variables on
$(0,1)$ independent of $X_{1}^{n} := \{ X_{1},\ldots,X_{n} \}$.
Pick and fix any $0 \leq j \leq m$. Let $\sigma_{1},\ldots,\sigma_{n}$
be independent Rademacher random variables independent of
$(U_{1},X_{1}),\ldots,\break(U_{n},X_{n})$.
Since $U_{1},\ldots,U_{n}$ are independent from $X_{1}^{n}$, applying
the symmetrization inequality [see Lemma 2.3.6 of \citet{VW96}]
conditional on $X_{1}^{n}$, we have
\[
\E \Bigl[ \sup_{u \in\U} \bigl( \En \bigl[ \bigl\{ u-1(U_{i}
\leq u) \bigr\} \hat{\eta}_{ij} \bigr] \bigr)^{2} \mid
X_{1}^{n} \Bigr] \leq4 \E \Bigl[ \sup_{u
\in\U} \bigl(
\En \bigl[ \sigma_{i} 1(U_{i} \leq u) \hat{
\eta}_{ij} \bigr] \bigr)^{2} \mid X_{1}^{n}
\Bigr].
\]
We make use of Proposition E.2 in Appendix E to bound the right side.
Consider the class of functions
\[
\mathcal{G} = \bigl\{ \R\times\R\ni(y,z) \mapsto1(y \leq u)z \dvtx u \in\U \bigr
\}.
\]
Then, we have
\[
\E \Bigl[ \sup_{u \in\U} \bigl( \En \bigl[ \sigma_{i}
1(U_{i} \leq u) \hat{\eta}_{ij} \bigr] \bigr)^{2}
\mid X_{1}^{n} \Bigr] = \E \Bigl[ \sup_{g
\in\mathcal{G}} \bigl(
\En \bigl[ \sigma_{i} g(U_{i},\hat{\eta}_{ij})
\bigr] \bigr)^{2} \mid X_{1}^{n} \Bigr].
\]
It is standard to see that $\mathcal{G}$ is a VC subgraph class with VC
index $\leq3$. Thus, by Theorem 2.6.7 of \citet{VW96}, there exist
universal constants $A \geq e$ and $W \geq1$ such that, for envelope
function $G(y,z) = | z |$,
\[
N \bigl( \epsilon\| G \|_{L_{2}(P_{n})}, \mathcal{G}, L_{2}(P_{n})
\bigr) \leq(A/\epsilon)^{W} \qquad 0 <\forall\epsilon\leq1,
\]
where $P_{n}$ denotes the empirical distribution on $\R\times\R$ that
assigns probability $n^{-1}$ to each $(U_{i},\hat{\eta}_{ij}),  i=1,\ldots,n$.
Therefore, by Proposition E.2, we conclude that
\[
\E \Bigl[ \sup_{u \in\U} \bigl( \En \bigl[ \sigma_{i}
1(U_{i} \leq u) \hat{\eta}_{ij} \bigr] \bigr)^{2}
\mid X_{1}^{n} \Bigr] \leq n^{-1}D \En \bigl[ \hat{
\eta}_{ij}^{2} \bigr],
\]
where $D$ is a universal constant. Since $0 \leq j \leq m$ is
arbitrary, we have
\[
\sup_{u \in\U} \bigl\| \En \bigl[ \bigl\{ u - 1 \bigl(Y_{i} \leq
Q_{Y|X}(u \mid X_{i}) \bigr) \bigr\} \hat{
\eta}_{i}^{m} \bigr] \bigr\|_{\ell^{2}} = O_{P}
\bigl\{ \bigl(\En \bigl[ \bigl\| \hat{\eta}_{i}^{m}
\bigr\|_{\ell^{2}}^{2} \bigr] \bigr)^{1/2} n^{-1/2}
\bigr\}.
\]
We shall show in Appendix D that
%
%e5.4 #&#
%
\begin{equation}
\En \bigl[ \bigl\| \hat{\eta}_{i}^{m} \bigr\|_{\ell^{2}}^{2}
\bigr] = O_{P}(m), \label{moment}
\end{equation}
by which we have
\[
\sup_{u \in\U} \bigl\| \En \bigl[ \bigl\{ u - 1 \bigl(Y_{i} \leq
Q_{Y|X}(u \mid X_{i}) \bigr) \bigr\} \hat{
\eta}_{i}^{m} \bigr]\bigr \|_{\ell^{2}} = O_{P}(
\sqrt{m/n}).
\]

Bounding $\mathit{II}$:
by Taylor's theorem, we have
\begin{eqnarray*}
&&F_{Y|X} \bigl(Q_{Y|X}(u \mid X) + y \mid X \bigr) -
F_{Y|X} \bigl(Q_{Y|X}(u \mid X) \mid X \bigr)
\\
&&\qquad= f_{Y|X} \bigl(Q_{Y|X}(u \mid X) \mid X \bigr) y +
y^{2} \int_{0}^{1}
f'_{Y|X} \bigl( Q_{Y|X}(u \mid X) + \theta y \mid
X \bigr) (1-\theta) \,d \theta
\\
&&\qquad=: f_{Y|X} \bigl(Q_{Y|X}(u \mid X) \mid X \bigr) y +
\frac{y^{2}}{2} R(u,y,X),
\end{eqnarray*}
by which we have, using (\ref{decomp}),
%
%e5.5 #&#
%
\begin{eqnarray}
\label{lbound} \mathit{II} &=&\En \bigl[ f_{Y|X} \bigl(Q_{Y|X}(u \mid
X_{i}) \mid X_{i} \bigr) \bigl(\hat{\eta}_{i}^{m}
\cdot\delta^{m} + \hat{r}_{i}(u) \bigr)
\bigl(h^{m} \cdot\hat{\eta}_{i}^{m} \bigr) \bigr]
\nonumber
\\
&&{} + \frac{1}{2} \En \bigl[ \bigl(\hat{\eta}_{i}^{m}
\cdot\delta^{m} + \hat{r}_{i} \bigr)^{2}R
\bigl(u, \hat{\eta}_{i}^{m} \cdot\delta^{m} +
\hat {r}_{i}(u),X_{i} \bigr) \bigl(h^{m} \cdot
\hat{ \eta}_{i}^{m} \bigr) \bigr]
\nonumber
\\
&\geq& M \sqrt{m/n} \En \bigl[ f_{Y|X} \bigl(Q_{Y|X}(u \mid
X_{i}) \mid X_{i} \bigr) \bigl(h^{m} \cdot\hat{
\eta}_{i}^{m} \bigr)^{2} \bigr]
\nonumber
\\
&&{} -C \En \bigl[ \bigl| \hat{r}_{i}(u) \bigl(h^{m} \cdot\hat{
\eta}_{i}^{m} \bigr) \bigr| \bigr] - C \En \bigl[ \bigl(\hat{
\eta}_{i}^{m} \cdot\delta^{m} +
\hat{r}_{i}(u) \bigr)^{2} \bigl| h^{m} \cdot\hat{
\eta}_{i}^{m} \bigr| \bigr]
\nonumber
\\[-8pt]
\\[-8pt]
\nonumber
&\geq& M \sqrt{m/n} \En \bigl[ f_{Y|X} \bigl(Q_{Y|X}(u
\mid X_{i}) \mid X_{i} \bigr) \bigl(h^{m} \cdot
\hat{ \eta}_{i}^{m} \bigr)^{2} \bigr]
\nonumber
\\
& &{}-C \bigl(\En \bigl[ \hat{r}^{2}_{i}(u) \bigr]
\bigr)^{1/2} \bigl(\En \bigl[ \bigl(h^{m} \cdot\hat{
\eta}_{i}^{m} \bigr)^{2} \bigr]
\bigr)^{1/2}
\nonumber
\\
&&{}- C M^{2} (m/n) \Bigl(\max_{1 \leq i \leq n} \bigl\| \hat{
\eta}_{i}^{m} \bigr\|_{\ell^{2}} \Bigr) \En \bigl[
\bigl(h^{m} \cdot\hat{\eta}_{i}^{m}
\bigr)^{2} \bigr]
\nonumber
\\
&&{} - C \Bigl( \max_{1 \leq i \leq n} \bigl\| \hat{\eta}_{i}^{m}
\bigr\|_{\ell^{2}} \Bigr) \En \bigl[\hat{r}^{2}_{i}(u) \bigr],\nonumber
\end{eqnarray}
where we have used the fact that $f_{Y|X}(y|X) \vee| f'_{Y|X}(y | X) |
\leq C$. By assumption~(A5), there exists a small constant $c_{1} > 0$
such that
\[
C \En \bigl[f_{Y|X} \bigl(Q_{Y|X}(u \mid X_{i})\bigr)
\bigl(h^{m} \cdot\hat{\eta}_{i}^{m}
\bigr)^{2} \bigr] \geq c_{1} \En \bigl[ \bigl(h^{m}
\cdot\hat{\eta}_{i}^{m} \bigr)^{2} \bigr] \qquad
\forall u \in\U, \forall h^{m} \in\mathbb{S}^{m}.
\]
We shall show in Appendix D that
%
%e5.6 #&#
%e5.7 #&#
%
\begin{eqnarray}
\sup_{h^{m} \in\mathbb{S}^{m}} \bigl|\En \bigl[ \bigl(h^{m} \cdot\hat {\eta
}_{i}^{m} \bigr)^{2} \bigr] - 1 \bigr| &=&
o_{P}(1), \label{matrix}
\\
\sup_{u \in\U} \En \bigl[\hat{r}^{2}_{i}(u) \bigr]& =&
O_{P} \bigl(mn^{-1} \bigr). \label{bias}
\end{eqnarray}
Thus, by (\ref{max}), (\ref{moment}), (\ref{matrix}) and (\ref{bias}),
we have
\[
(\ref{lbound}) \geq c_{1}M \sqrt{m/n} \bigl(1 - o_{P}(1)
\bigr) - O_{P}(\sqrt{m/n}) - M^{2} o_{P}(
\sqrt{m/n}),
\]
where the stochastic orders are evaluated uniformly in $u \in\U$ and
$h^{m} \in\mathbb{S}^{m}$.

Bounding $\mathit{III}$:
Let $\sigma_{1},\ldots,\sigma_{n}$ be independent Rademacher random
variables independent of the data $(Y_{1},X_{1}),\ldots,(Y_{n},X_{n})$.
Applying the symmetrization inequality conditional on $X_{1}^{n} := \{
X_{1},\ldots,X_{n} \}$, we have
%
%e5.8 #&#
%
\begin{eqnarray}
\label{exp} & &\E \Bigl[ \sup_{u \in\U, h^{m} \in\mathbb{S}^{m}} \bigl| n^{-1/2}
\mathbb{G}_{n | X} \bigl[ \bigl\{ 1 \bigl(Y_{i} \leq\hat{
\eta}_{i}^{m} \cdot \bigl(d^{m}(u) +
\delta^{m} \bigr) \bigr)
\nonumber
\\
&&\hspace*{95pt}\qquad{} - 1 \bigl(Y_{i} \leq Q_{Y|X}(u \mid
X_{i}) \bigr) \bigr\} \bigl(h^{m} \cdot\hat{
\eta}_{i}^{m} \bigr) \bigr] \bigr| \mid X_{1}^{n}
\Bigr]
\nonumber
\\[-8pt]
\\[-8pt]
\nonumber
&&\quad\leq2 \E \Bigl[ \sup_{u \in\U, h^{m} \in\mathbb{S}^{m}} \bigl| \En \bigl[ \sigma_{i}
\bigl\{ 1 \bigl(Y_{i} \leq\hat{\eta}_{i}^{m}
\cdot \bigl(d^{m}(u) + \delta^{m} \bigr) \bigr)
\nonumber
\\
&&\hspace*{96pt}\qquad{} - 1 \bigl(Y_{i} \leq Q_{Y|X}(u \mid
X_{i}) \bigr) \bigr\} \bigl(h^{m} \cdot\hat{
\eta}_{i}^{m} \bigr) \bigr] \bigr| \mid X_{1}^{n}
\Bigr],\nonumber
\end{eqnarray}
where $\delta^{m}$ is taken as $\delta^{m} = M\sqrt{m/n} h^{m}$ in the
suprema. Note that the symmetrization inequality is applicable since
the regular conditional distribution of $(Y_{1},\ldots,Y_{n})'$ given
$X_{1}^{n}$ exists and conditional on $X_{1}^{n}$, $Y_{1},\ldots,Y_{n}$
are independent. Consider the class of functions
\begin{eqnarray*}
\mathcal{G} &=& \bigl\{ \R\times D[0,1] \times\R^{m+1} \ni \bigl(y,x,
\eta^{m} \bigr) \mapsto \bigl\{ 1 \bigl(y \leq\eta^{m} \cdot
\bigl(d^{m}(u) + \delta^{m} \bigr) \bigr)
\\
&&\hspace*{150pt}\qquad{}-1 \bigl(y \leq Q_{Y|X}(u \mid x) \bigr) \bigr\}
\bigl(h^{m} \cdot\eta^{m} \bigr) \dvtx\\
&&\hspace*{175pt}{} u \in\U,
h^{m} \in \mathbb{S}^{m}, \delta^{m} = M
\sqrt{m/n} h^{m} \bigr\}.
\end{eqnarray*}
Then, we have
\[
(\ref{exp}) = 2 \E \Bigl[ \sup_{g \in\mathcal{G}} \bigl| \En \bigl[
\sigma_{i} g \bigl(Y_{i},X_{i},\hat{
\eta}_{i}^{m} \bigr) \bigr] \bigr| \mid X_{1}^{n}
\Bigr].
\]
We apply Proposition E.1 in Appendix E to bound the right side.
Note that $(X_{i},\hat{\eta}_{i}^{m})$ are measurable with respect to
the $\sigma$-field generated by $X_{1}^{n}$, the regular conditional
distribution of $(Y_{1},\ldots,Y_{n})'$ given $X_{1}^{n}$ exists, and
conditional on $X_{1}^{n}$, $Y_{1},\ldots,Y_{n}$ are independent.
Observe that
\[
\sup_{g \in\mathcal{G}} \bigl| g \bigl(Y_{i},X_{i},\hat{
\eta}_{i}^{m} \bigr) \bigr| \leq\max_{1 \leq i \leq n} \bigl\| \hat{
\eta}_{i}^{m} \bigr\|_{\ell^{2}} =: \hat{B},
\]
and, by (\ref{decomp}),
\begin{eqnarray*}
&&\sup_{g \in\mathcal{G}} \En \bigl[ \E \bigl[ g^{2}
\bigl(Y_{i}, X_{i}, \hat{\eta}_{i}^{m}
\bigr) \mid X_{1}^{n} \bigr] \bigr]
\\
&&\qquad= \sup_{u \in\U, h^{m} \in\mathbb{S}^{m}} \En \bigl[ \bigl| F_{Y|X} \bigl( \hat{
\eta}_{i}^{m} \cdot \bigl(d^{m}(u) +
\delta^{m} \bigr) \mid X_{i}\bigr)
\nonumber
\\
&&\hspace*{64pt}\qquad\quad{} - F_{Y|X} \bigl(Q_{Y|X}(u \mid
X_{i}) \mid X_{i} \bigr) \bigr| \bigl(h^{m} \cdot
\hat{\eta}_{i}^{m} \bigr)^{2} \bigr]
\\
&&\qquad\leq\sup_{u \in\U, h^{m} \in\mathbb{S}^{m}} \bigl\{ C M\sqrt {m/n} \En \bigl[ \bigl|
h^{m} \cdot\hat{\eta}_{i}^{m}\bigr |^{3}
\bigr] + C \En \bigl[ \bigl| \hat{r}_{i}(u) \bigr| \bigl(h^{m} \cdot
\hat{ \eta}_{i}^{m} \bigr)^{2} \bigr] \bigr\}
\\
&&\qquad\leq C M \sqrt{m/n} \max_{1 \leq i \leq n} \bigl\| \hat{\eta}_{i}^{m}
\bigr\|_{\ell^{2}} \sup_{h^{m} \in\mathbb{S}^{m}} \En \bigl[ \bigl(h^{m} \cdot
\hat{\eta}_{i}^{m} \bigr)^{2} \bigr]
\\
&&\qquad\quad{} + C \max_{1 \leq i \leq n} \bigl\| \hat{\eta}_{i}^{m}
\bigr\|_{\ell^{2}} \Bigl(\sup_{u \in
\U} \En \bigl[ \hat{r}_{i}^{2}(u)
\bigr] \Bigr)^{1/2} \Bigl(\sup_{h^{m} \in\mathbb{S}^{m}} \En \bigl[
\bigl(h^{m} \cdot\hat{\eta}_{i}^{m}
\bigr)^{2} \bigr] \Bigr)^{1/2}
\\
&&\qquad=: \hat{\tau}^{2}.
\end{eqnarray*}
We shall show in Appendix D
that there exist some constants $c_{2} \geq1$ and $A' \geq3 \sqrt{e}$
such that
%
%e5.9 #&#
%
\begin{equation}
N \bigl(\hat{B} \epsilon, \mathcal{G}, L_{2} \bigl(P'_{n}
\bigr) \bigr) \leq \bigl(A'/\epsilon \bigr)^{c_{2}m}, \qquad 0
< \forall\epsilon\leq1, \label{covnumber}
\end{equation}
where $P'_{n}$ denotes the empirical distribution on $\R\times D[0,1]
\times\R^{m+1}$ that assigns probability $n^{-1}$ to each
$(Y_{i},X_{i},\hat{\eta}_{i}^{m}),  i=1,\ldots,n$.\vadjust{\goodbreak}
Therefore, by Proposition E.1 in Appendix E, we conclude that
%
%e5.10 #&#
%
\begin{eqnarray}
\label{ebound} &&\E \Bigl[ \sup_{g \in\mathcal{G}} \bigl| \En \bigl[
\sigma_{i} g \bigl(Y_{i},X_{i},\hat{
\eta}_{i}^{m} \bigr) \bigr] \bigr| \mid X_{1}^{n}
\Bigr]
\nonumber
\\[-8pt]
\\[-8pt]
\nonumber
&&\qquad\leq1(\hat{\tau} > 0) D' \biggl[ \sqrt{
\frac{c_{2} m \hat{\tau
}^{2}}{n}} \sqrt{ \log\frac{A' \hat{B}}{\hat{\tau}}} + \frac{c_{2} m
\hat{B}}{n} \log
\frac{ A' \hat{B}}{\hat{\tau}} \biggr],
\end{eqnarray}
provided that $\hat{\tau} \leq\hat{B}$, where $D'$ is a universal constant.

By (\ref{max}), (\ref{moment}) and (\ref{matrix}), and the fact that
$M>1$, we have
\[
\hat{B} = o_{P} \bigl\{ (\log n)^{-1} \sqrt{n/m} \bigr\},
\qquad \hat{\tau}^{2} = M o_{P} \bigl\{ (\log
n)^{-1} \bigr\},
\]
and there exists a small constant $c_{3} > 0$ such that with
probability approaching one
\[
\hat{\tau}^{2} \geq c_{3} \hat{B} \sqrt{m/n}.
\]
Thus, replacing $\hat{\tau}$ by $\hat{\tau} \wedge\hat{B}$ if
necessary, $(\ref{ebound}) = M^{1/2} o_{P}(\sqrt{m/n})$, by which we
conclude that
\[
\mathit{III} \geq- M^{1/2} o_{P}(\sqrt{m/n}),
\]
where the stochastic order is evaluated uniformly in $u \in\U$ and
$h^{m} \in\mathbb{S}^{m}$.

Taking these together, we now conclude that
\[
I+\mathit{II}+\mathit{III} \geq c_{1} M \sqrt{m/n} \bigl(1-o_{P}(1) \bigr) -
O_{P}(\sqrt{m/n}) - M^{2} o_{P} (\sqrt{m/n}),
\]
where the stochastic orders are evaluated uniformly in $u \in\U$ and
$h^{m} \in\mathbb{S}^{m}$. This immediately implies the hypothesis of
Lemma~\ref{lemred}.

%s5.3 #&#
\subsection{Completion of the proof}\label{sec5.3}

We have shown that
\[
\sup_{u \in\U} \bigl\| \hat{d}^{m}(u) - d^{m}(u)
\bigr\|_{\ell^{2}} = O_{P}(\sqrt{m/n}).
\]
Since $\| \hat{d}^{m}(u) - d^{m}(u) \|_{\ell^{2}}^{2} = \sum_{j=0}^{m}
\kappa_{j} (\hat{b}_{j}(u) -  b_{j}(u))^{2} \geq\kappa_{m} \sum_{j=0}^{m} (\hat{b}_{j}(u) - b_{j}(u))^{2} = \kappa_{m} \| \hat
{b}^{m}(u) - b^{m}(u) \|_{\ell^{2}}^{2}$ (with $\kappa_{0} := 1$), we have
\begin{eqnarray*}
\sup_{u \in\U} \bigl\| \hat{b}^{m}(u) - b^{m}(u)
\bigr\|_{\ell^{2}}^{2} &=& O_{P} \bigl(\kappa_{m}^{-1}
m n^{-1} \bigr) = O_{P} \bigl(m^{\alpha+1}n^{-1}
\bigr)
\\
&=& O_{P} \bigl(n^{-(2\beta-1)/(\alpha+2\beta)} \bigr).
\end{eqnarray*}
Observe that\vspace*{-1pt}
\begin{eqnarray*}
&&\hat{b}(t,u) - b(t,u)
\\[-1pt]
&&\qquad= \sum_{j=1}^{m}
\hat{b}_{j}(u) \hat{\phi}_{j}(t) - \sum
_{j=1}^{m} b_{j}(u) \hat{
\phi}_{j}(t) + \sum_{j=1}^{m}
b_{j}(u) \hat{ \phi}_{j}(t)
\\[-1pt]
& &\qquad\quad{}- \sum_{j=1}^{m}
b_{j}(u) \phi_{j}(t) - \sum_{j=m+1}^{\infty}
b_{j}(u) \phi_{j}(t)
\\[-1pt]
&&\qquad=\sum_{j=1}^{m} (
\hat{b}_{j} - b_{j}) (u) \hat{\phi}_{j}(t) + \sum
_{j=1}^{m} b_{j}(u) \bigl(\hat{
\phi}_{j}(t) - \phi_{j}(t) \bigr) - \sum
_{j=m+1}^{\infty} b_{j}(u) \phi_{j}(t),\vspace*{-1pt}
\end{eqnarray*}
so that, uniformly in $u \in\U$,\vspace*{-1pt}
\begin{eqnarray*}
&&\int_{0}^{1} \bigl(\hat{b}(t,u) - b(t,u)
\bigr)^{2} \,dt
\\[-1pt]
&&\qquad\leq3 \bigl\| \hat{b}^{m}(u) - b^{m}(u)
\bigr\|_{\ell^{2}}^{2} + 3m \sum_{j=1}^{m}
b_{j}^{2}(u) \| \hat{\phi}_{j} -
\phi_{j} \|^{2} + 3 \sum_{j=m+1}^{\infty}
b_{j}^{2} (u)
\\[-1pt]
&&\qquad=O_{P} \bigl(n^{-(2\beta-1)/(\alpha+2\beta)} \bigr) + 3m \sum
_{j=1}^{m} b_{j}^{2} (u) \|
\hat{\phi}_{j} - \phi_{j} \|^{2} + O
\bigl(m^{-2\beta+1} \bigr)
\\[-1pt]
&&\qquad=O_{P} \bigl(n^{-(2\beta-1)/(\alpha+2\beta)} \bigr) + 3m \sum
_{j=1}^{m} b_{j}^{2} (u) \|
\hat{\phi}_{j} - \phi_{j} \|^{2}.
\end{eqnarray*}
By Lemmas
D.1 and D.2 together with the computation in the proof of (\ref{bias}) in Appendix D, we see that\vspace*{-1pt}
\begin{eqnarray*}
m \sum_{j=1}^{m} b_{j}^{2}(u)
\| \hat{\phi}_{j} - \phi_{j} \|^{2} &\leq& C m
\sum_{j=1}^{m} j^{-2\beta} \| \hat{
\phi}_{j} - \phi_{j} \|^{2}
\\
& = &O_{P} \bigl(m \bigl(n^{-1} + \Delta^{\gamma} +
n^{-1}(\log n) m^{3\alpha+3} \Delta^{\gamma} \bigr) \bigr) = O
\bigl(mn^{-1} \bigr)
\\
&=& o_{P} \bigl(n^{-(2\beta-1)/(\alpha+2\beta)} \bigr).
\end{eqnarray*}
This completes the proof of Theorem~\ref{thm1}.

\section*{Acknowledgments}
Most of this work was done while the author was visiting the Department
of Economics, MIT. He would like to thank Professors Victor
Chernozhukov, Hidehiko Ichimura and Roger Koenker for their suggestions
and encouragements.
He also would like to thank the Editor, Professor T. Tony Cai, the
Associate Editor and three anonymous referees for constructive comments
that helped to improve the quality of the paper.

\begin{supplement}%[id=suppA]
\stitle{Supplement to ``Estimation in functional linear quantile regression''}\label{suppA}
\slink[doi]{10.1214/12-AOS1066SUPP} %[doi,text={...}] - jei reikia suskaldyti doi
\sdatatype{.pdf}
\sfilename{aos1066\_supp.pdf}
\sdescription{This supplementary file contains the additional
discussion on the connection to nonlinear ill-posed inverse problems,
technical proofs omitted in the main body, some useful technical tools
and additional simulation results.}
\end{supplement}

\printaddresses

\end{document}